\documentclass[final,leqno,showlabe]{siamltex}
\usepackage{amsmath}
\usepackage{amssymb}
\usepackage{cases}
\usepackage{enumerate}
\usepackage{graphicx}
\usepackage{cleveref}
\usepackage[usenames]{color}
\usepackage{wrapfig}
\usepackage{stmaryrd}
\usepackage{comment}
\SetSymbolFont{stmry}{bold}{U}{stmry}{m}{n}

\newtheorem{remark}{Remark}[section]

\title{An energy-stable parametric finite element method for  the planar Willmore flow}

\author{Weizhu Bao\thanks{Department of Mathematics, National University of Singapore, Singapore 119076 (matbaowz@nus.edu.sg, URL:
http://blog.nus.edu.sg/matbwz/). }
\and Yifei Li\thanks{Department of Mathematics, National University of
Singapore, Singapore, 119076 (liyifei@nus.edu.sg)}
}

\date{}
\begin{document}

\maketitle
%%%%% Begin Abstract %%%%%%%%%%%

\begin{abstract}
We propose an energy-stable parametric finite element method (PFEM) for the planar Willmore flow and establish its unconditional energy stability of the full discretization scheme. The key lies in the introduction of two novel geometric identities to describe the planar Willmore flow: the first one involves the coupling of the outward unit normal vector $\boldsymbol{n}$ and the normal velocity $V$, and the second one concerns the time derivative of the mean curvature $\kappa$. Based on them, we derive a set of new geometric partial differential equations for the planar Willmore flow, leading to our new fully-discretized and unconditionally energy-stable PFEM. Our stability analysis is also based on the two new geometric identities. Extensive numerical experiments are provided to illustrate its efficiency and validate its unconditional energy stability. 
\end{abstract}
%%%%% end %%%%%%%%%%%

%%%%% Keywords %%%%%%%%%%%
%\pac{}
%\ams{35Q55, 65M70, 65N25, 65N35, 81Q05}

%\pac{68.35.-p, 68.55.Jk, 68.37.-d, 81.16.Rf}



\pagestyle{myheadings} \markboth{W. Bao, and Y. Li}
{An Energy-stable PFEM for Planar Willmore Flow}

\section{Introduction}
\label{intro}
The Willmore energy is an important conformal invariant in differential geometry studied by Willmore \cite{thomsen1923konforme,willmore1965note,willmore1993riemannian}, which quantifies the deviation of a surface from a sphere. It is defined as the integral of the squared mean curvature over the surface, thus providing a measure of surface elasticity and is essential in describing surface elastic energy \cite{germain1821recherches,bryant1986reduction}. The critical points of the  Willmore energy, known as Willmore surfaces \cite{willmore1965note}, have wide applications in the understanding of lipid bilayers \cite{canham1970minimum,helfrich1973elastic}, modelling of optimal surfaces \cite{lott1988method,dyn2001optimizing,https://doi.org/10.1111/1467-8659.1330143}, design of optical lenses \cite{santosa2003analysis}, and the regularization energy in geometric flows \cite{Bao17,chen2018efficient}.  The broad applications of the Willmore energy and its corresponding Willmore surfaces extend to biomembranes \cite{seifert1997configurations,bonito2011dynamics,boudaoud2010introduction}, geometric modeling \cite{bohle2008constrained}, materials science \cite{grinspun2003discrete}, string theory \cite{Ginsparg:1993is}, and computational geometry \cite{gruber2020computational,bretin2011regularization}.

Let $\Gamma\subset\mathbb{R}^2$ be a closed planar curve in $\mathbb{R}^2$ with a unit outward normal vector $\boldsymbol{n}$. The Willmore energy $W(\Gamma)$ is formulated as \cite{Deckelnick05}
\begin{equation}\label{eq: def of willmore energy}
    W(\Gamma):=\frac{1}{2}\int_{\Gamma} \kappa^2 \, dA,
\end{equation}
where $\kappa$ is the mean curvature of $\Gamma$, and $dA$ denotes the area element. 

The Willmore flow is a natural approach to study the Willmore energy and its corresponding Willmore surfaces. It can be regarded as the $L^2$-gradient flow of the Willmore energy $W(\Gamma)$ \cite{dziuk2002evolution,Barrett08}, which dissipates the Willmore energy and drives a surface evolving towards the Willmore surface. The mathematical formulation of the Willmore flow for an evolving closed planar curve $\Gamma:=\Gamma(t)\subset\mathbb{R}^2$ is given as \cite{barrett2020parametric}:
\begin{equation}\label{eq: def of willmore flow}
    V=\partial_{ss}\kappa +\frac{1}{2} \kappa^3,
\end{equation} 
where $V$ denotes the normal velocity, $s$ is the arc length parameter, and $\partial_s$ represents the arc length derivative with respect to $\Gamma$. It is a fourth-order, highly nonlinear geometric partial differential equation. And its solution satisfies the following energy dissipative property \cite{barrett2020parametric}:
\begin{equation}\label{eq: energy dissipation, continuous}
  \frac{d W(\Gamma(t))}{dt} = -\int_{\Gamma(t)} V^2\, dA\quad\Longrightarrow \quad   W(\Gamma(t))\leq W(\Gamma(t'))\leq W(\Gamma(0)), \qquad  \forall t\geq t'\geq 0.
\end{equation}
Theoretical results of the Willmore energy and the corresponding Willmore flow can be found in \cite{langer1984total,langer1985curve,simon1993existence,kuwert2002gradient,bauer2003existence,
dall2014willmore,wen1995curve} and references therein. %For numerical approximations, it is desirable to preserve the energy dissipation \eqref{eq: energy dissipation, continuous} at the full-discretized level. 

Based on different descriptions of the evolving curve $\Gamma(t)$, many different numerical schemes have been proposed and analyzed for simulating the Willmore flow.  For implicitly captured curves, Droske and Rumpf proposed a level set approach \cite{rumpf2004level}, and the phase-field approach was developed by Du et al. \cite{du2005phase}. For explicit curves, Mayer and Simonett first proposed a finite difference scheme for axially symmetric surfaces \cite{mayer2002numerical}, which provides numerical evidence for the formation of singularities in the Willmore flow. Using the graph formulation, where $\Gamma(t)$ is represented as a height function, Deckelnick and Dziuk first proposed a convergent finite element method with $C^0$ finite elements \cite{dziuk2006error}, and later with $C^1$ finite elements \cite{deckelnick2015C1}. For the Willmore flow of curves in arbitrary codimensions, Dziuk, Kuwert, and Schatzle proposed a finite element method and proved the long-term existence of its solutions \cite{dziuk2002evolution}. For additional related works, we refer to \cite{bretin2015phase,bobenko2005discrete,Deckelnick05,barrett2020parametric} and the references therein. 

Compared to the aforementioned methods, the parametric methods offer a more direct and efficient representation of $\Gamma(t)$ in the computation of the Willmore flow. Rusu first proposed a linear finite element approach \cite{rusu2005algorithm} based on a mixed variational formulation by introducing a mean curvature vector. In \cite{kovacs2021convergent}, Kov{\'a}cs, Li, and Lubich applied the evolving surface finite element method to the Willmore flow and proved that this method admits a convergent analysis of higher order finite elements. Barrett, Garcke, and N{\"u}rnberg developed a parametric finite element method (PFEM) \cite{Barrett08} with a special treatment of the mean curvature, which ensures a good mesh quality. However, parametric methods often involve complex geometric identities to handle mean curvature. Due to the high nonlinearity of the Willmore energy, proving the energy stability of parametric methods is notoriously challenging, even at the semi-discrete level.

A breakthrough was achieved by Dziuk, who proposed an essentially new parametric method in \cite{dziuk2008computational} that is capable of proving the energy dissipation in its semi-discretization. Barrett, Garcke, and N{\"u}rnberg highlighted that the new geometric partial differential equation (PDE) for the Willmore flow, introduced in Dziuk's parametric method, enables the proof of energy stability. Building on this understanding, they developed a new parametric approach \cite{barrett2016computational} that not only preserves energy dissipation in semi-discretization but also exhibits good mesh properties. Furthermore, they extended their parametric approach to other geometric flows related to the Willmore energy \cite{barrett2017stable}. The adaptation of Dziuk's geometric PDE enables these extensions to be energy-stable at the semi-discretized level. In fact, the energy stability of all the aforementioned parametric 
methods can only be established in the semi-discretization in space. Unfortunately, the energy stability cannot be
extended to full-discretization. To our best knowledge, the energy stability of the parametric numerical schemes for the Willmore flow at the full-discretized level remains open, i.e. no existing parametric numerical scheme is being proved mathematically energy-stable at the full-discretized level. 

The objective of this paper is to develop an energy-stable parametric finite element method for the Willmore flow of planar curves, which can be proved to be unconditionally energy-stable at the fully-discretized level.
The key idea is to find a new and proper geometric PDE for the Willmore flow \eqref{eq: def of willmore flow}. The new geometric PDE is composed of two parts -- the first part is a mixed differential-algebraic system for the evolving planar curve $\Gamma(t)$ and the normal velocity $V$; and the second part is a new evolution equation for the mean curvature $\kappa$. The first part is motivated by our work on surface diffusion \cite{li2021energy,bao2023symmetrized} and related geometric flows \cite{bao2022structure}. We discover that, for analyzing stability, especially at the fully-discretized level, a weak vector equation for $V \boldsymbol{n}$ (coupling $V$ with $\boldsymbol{n}$) is more effective than a strong scalar equation for the normal velocity $V$. The motivation for the second part comes from Kov{\'a}cs, Li, and Lubich's work on the mean curvature flow \cite{kovacs2019convergent} and the Willmore flow \cite{kovacs2021convergent}. These works suggest treating the mean curvature $\kappa$ as an independently evolving unknown.

The numerical scheme and the stability analysis are motivated by the PFEMs developed by Barrett, Garcke, and N{\"u}rnberg \cite{Barrett07,Barrett08,barrett2020parametric}, as well as our own recent development in PFEMs \cite{li2021energy,bao2021structure,bao2023symmetrized}. We have realized that, in addition to achieving good mesh properties, the full discretization using PFEM typically results in a conformal discretization of the original geometric PDE. This means that the full discretization inherits the energy dissipation property of the continuous/variational equation.

The paper is organized as follows. In Section 2, we derive a variational formulation of the Willmore flow by introducing our new geometric PDE, and then prove its energy stability. Section 3 describes the semi-discretization in space of the weak formulation. We propose the energy-stable PFEM full-discretization and analyze its  stability in Section 4. Section 5 presents extensive numerical results to demonstrate the performance and to validate the stability of our energy-stable PFEM. And finally, we conclude in Section 6.

\section{A new geometric PDE and its variational formulation}

Let $\Gamma:=\Gamma(t)\subset\mathbb{R}^2$ be the evolving planar curve, starting with the initial planar curve $\Gamma(0)=\Gamma_0$. The planar curve $\Gamma(t)$ is parameterized by $\boldsymbol{X}(\mathbb{T}, t)$ as
\begin{equation}\label{eq: def of X}
    \boldsymbol{X}(\cdot, t): \mathbb{T}\to \mathbb{R}^2, \quad (\rho, t)\mapsto \boldsymbol{X}(\rho, t)=(x(\rho, t), y(\rho, t))^T,
\end{equation}
where $\mathbb{T}:=\mathbb{R}/\mathbb{Z}=[0, 1]$ is the periodic unit interval, and $\boldsymbol{X}(\cdot, 0):=\boldsymbol{X}_0(\cdot)$ is a parameterization of the initial planar curve $\Gamma_0$. From this parameterization, we derive the arc length parameter $s$ as $s(\rho, t):=\int_0^{\rho} |\partial_q \boldsymbol{X}| dq$. The arc length derivative $\partial_s$ and the arc length element $ds$ are then defined as $\partial_s :=\frac{1}{|\partial\rho \boldsymbol{X}|} \partial_\rho$ and $ds:=\partial_\rho s \,d\rho = |\partial_\rho \boldsymbol{X}|\, d\rho$, respectively. Moreover, the unit tangent vector $\boldsymbol{\tau}$ and the outward unit normal vector $\boldsymbol{n}$ are determined by
\begin{equation}\label{eq: def of tau and n}
    \boldsymbol{\tau}:=\boldsymbol{\tau}(\rho, t)=\partial_s \boldsymbol{X}(\rho, t)=\frac{\partial_\rho\boldsymbol{X}(\rho, t)}{|\partial_\rho\boldsymbol{X}(\rho, t)|}, \quad \boldsymbol{n}:=\boldsymbol{n}(\rho, t)=-\boldsymbol{\tau}^\perp,
\end{equation}
where $\perp$ denotes the clockwise rotation by $\frac{\pi}{2}$, such that for any vector $\boldsymbol{u}=(u_1, u_2)^T\in \mathbb{R}^2$, $\boldsymbol{u}^{\perp}=(-u_2, u_1)^T$. The curvature $\kappa$ is then given by
\begin{equation}\label{eq: def of kappa 1}
    \kappa:=\kappa(\rho, t)=-\boldsymbol{n}\cdot \partial_{ss}\boldsymbol{X}(\rho, t) = -\boldsymbol{n}\cdot \frac{1}{|\partial_\rho \boldsymbol{X}|}\partial_\rho \left(\frac{\partial_\rho \boldsymbol{X}}{|\partial_\rho \boldsymbol{X}|}\right).
\end{equation}

This parameterization allows us to reformulate the Willmore flow \eqref{eq: def of willmore flow} into a fourth-order geometric PDE for $\boldsymbol{X}(\rho, t)$, as follows:
\begin{subequations}
\label{eq: willmore pde origin}
\begin{align}
\label{eq: willmore pde origin 1}
&\partial_t \boldsymbol{X} = V \boldsymbol{n},\\
\label{eq: willmore pde origin 2}
&V = \partial_{ss} \kappa +\frac{1}{2}\kappa^3,\\
\label{eq: willmore pde origin 3}
&\kappa = -\boldsymbol{n}\cdot \partial_{ss}\boldsymbol{X}, \qquad \forall \rho \in \mathbb{T}, \ t>0,
\end{align}
\end{subequations}
with the initial condition $\boldsymbol{X}(\rho, 0)=\boldsymbol{X}_0(\rho)$ and the outward unit normal $\boldsymbol{n}$ defined as in \eqref{eq: def of tau and n}.

Suppose $\Gamma(t)=\boldsymbol{X}(\mathbb{T}, t)$ is a solution to \eqref{eq: willmore pde origin}. From \eqref{eq: def of willmore energy}, its Willmore energy $W(t)$ at time $t$, is given as
\begin{equation}\label{eq: def of willmore energy, continuous}
    W(t):=W(\Gamma(t))=\int_{\Gamma(t)} \kappa^2 \, ds = \int_0^1 \kappa^2(\rho, t)|\partial_\rho \boldsymbol{X}(\rho, t)|\, d\rho.
\end{equation}

\subsection{A new geometric PDE}
Here we establish two key geometric identities, the first one is specified for the normal velocity $V$ of the Willmore flow, and the second one is a new formulation of $\kappa$, which holds for all geometric flows. These identities play a decisive role in establishing the energy-stable discretizations for the Willmore flow.
\begin{lemma}\label{lem: geo identity for V}
For a solution $\boldsymbol{X}$ of the Willmore flow \eqref{eq: willmore pde origin}, the normal velocity $V$ satisfies the following geometric identity:
\begin{equation}\label{eq: key identity 1}
    V \boldsymbol{n}=\partial_s\left(\partial_s \kappa \boldsymbol{n}-\frac{1}{2}\kappa^2\partial_s \boldsymbol{X}\right).
\end{equation}
\end{lemma}
\begin{proof}
Since $\boldsymbol{\tau}$ is the unit tangent and $\boldsymbol{n}$ is the unit normal, we know that $\boldsymbol{n}\cdot \boldsymbol{n}=1, \boldsymbol{\tau}\cdot \boldsymbol{n}=0$ hold for all $\rho \in \mathbb{T}$. Consequently, we have
\begin{equation}\label{eq: lem1, aux 1}
    \boldsymbol{n}\cdot \partial_s \boldsymbol{n}=0, \quad \boldsymbol{n}\cdot \partial_s \boldsymbol{\tau}+\boldsymbol{\tau}\cdot \partial_s \boldsymbol{n} = 0.
\end{equation}
Moreover, applying vector decomposition to $\partial_s \boldsymbol{n}$ on the orthonormal basis ${\boldsymbol{\tau}, \boldsymbol{n}}$, together with \eqref{eq: lem1, aux 1} and the definition of the curvature $\kappa$ in  \eqref{eq: def of kappa 1}, we find that
\begin{equation}\label{eq: lem1, aux 2}
    \partial_s \boldsymbol{n} = (\boldsymbol{n}\cdot \partial_s \boldsymbol{n})\boldsymbol{n} + (\boldsymbol{\tau}\cdot \partial_s \boldsymbol{n})\boldsymbol{\tau}=-(\boldsymbol{n}\cdot \partial_s\boldsymbol{\tau})\boldsymbol{\tau}=\kappa\, \partial_s \boldsymbol{X}.
\end{equation}
We can therefore simplify the right-hand side of \eqref{eq: key identity 1} as
\begin{align}\label{eq: lem1, aux 3}
    &\partial_s\left(\partial_s \kappa \,\boldsymbol{n}-\frac{1}{2}\kappa^2\,\partial_s \boldsymbol{X}\right)\nonumber\\
    &=\partial_{ss}\kappa\, \boldsymbol{n}+\partial_s\kappa\, \partial_s \boldsymbol{n}-\kappa\, \partial_s \kappa \, \partial_s \boldsymbol{X}-\frac{1}{2}\kappa^2\, \partial_{ss}\boldsymbol{X}\nonumber\\
    &=\partial_{ss}\kappa\, \boldsymbol{n}+\partial_s\kappa \,(\kappa \,\partial_s\boldsymbol{X})-\kappa\, \partial_s\kappa\, \partial_s\boldsymbol{X} -\frac{1}{2}\kappa^2\,(-\kappa\boldsymbol{n})\nonumber\\
    &=\left(\partial_{ss}\kappa+\frac{1}{2} \kappa^3\right)\boldsymbol{n}.
\end{align}
This, together with \eqref{eq: willmore pde origin 2} implies the desired geometric identity \eqref{eq: key identity 1}. 
\end{proof}

Next, we examine the time derivative of the curvature $\kappa$ in the following lemma.

\begin{lemma}\label{lem: geo identity for kappa}
For the time derivative of the curvature $\kappa$, we have the following geometric identity:
\begin{equation}\label{eq: key identity 2}
    \partial_t \kappa = -\partial_{s}(\boldsymbol{n}\cdot \partial_s\left(\partial_t \boldsymbol{X}\right))-(\partial_s \boldsymbol{X}\cdot \partial_s\left(\partial_t \boldsymbol{X}\right))\, \kappa.
\end{equation}
\end{lemma}
\begin{proof}We begin by applying vector decomposition to $\partial_{ss}\boldsymbol{X}$ and considering the identity  $0=\boldsymbol{\tau}\cdot \partial_s\boldsymbol{\tau}=\boldsymbol{\tau}\cdot \partial_{ss}\boldsymbol{X}$ from \eqref{eq: lem1, aux 1} along with the definition of $\kappa$ in \eqref{eq: def of kappa 1}. This leads to
\begin{equation}\label{eq: lem2, aux 0}
    \partial_{ss}\boldsymbol{X} = (\boldsymbol{n}\cdot \partial_{ss}\boldsymbol{X})\boldsymbol{n}+(\boldsymbol{\tau}\cdot \partial_{ss}\boldsymbol{X})\boldsymbol{\tau}=-\kappa\, \boldsymbol{n}.
\end{equation}
Next, we consider the time derivative of $|\partial_\rho \boldsymbol{X}|=\sqrt{(\partial_\rho x)^2+(\partial_\rho y)^2}$. Utilizing the chain rule, we get
\begin{align}\label{eq: lem2, aux 1}
\partial_t (|\partial_\rho \boldsymbol{X}|) &= \frac{\partial_t(\partial_\rho x)\partial_\rho x+\partial_t(\partial_\rho y)\partial_\rho y}{\sqrt{(\partial_\rho x)^2+(\partial_\rho y)^2}}\nonumber\\
&=|\partial_\rho \boldsymbol{X}|\,\frac{\partial_\rho \boldsymbol{X}}{|\partial_\rho \boldsymbol{X}|}\cdot\frac{\partial_\rho (\partial_t \boldsymbol{X})}{|\partial_\rho \boldsymbol{X}|} \nonumber\\
&=|\partial_\rho \boldsymbol{X}|\,\partial_s\boldsymbol{X}\cdot \partial_s (\partial_t \boldsymbol{X}).
\end{align}
The commutator of $\partial_t$ and $\partial_s$ can be thus simplified as
\begin{align}\label{eq: lem2, aux 2}
\partial_t(\partial_s) - \partial_s (\partial_t) &= \partial_t \left(\frac{\partial_\rho}{|\partial_\rho \boldsymbol{X}|}\right)-  \frac{\partial_\rho(\partial_t)}{|\partial_\rho \boldsymbol{X}|}\nonumber\\
&=\frac{|\partial_\rho \boldsymbol{X}|\partial_t(\partial_\rho) - \partial_t (|\partial_\rho \boldsymbol{X}|)\partial_\rho}{|\partial_\rho \boldsymbol{X}|^2}-  \frac{\partial_\rho(\partial_t)}{|\partial_\rho \boldsymbol{X}|}\nonumber\\
&=-\frac{|\partial_\rho \boldsymbol{X}|\,\partial_s\boldsymbol{X}\cdot \partial_s (\partial_t \boldsymbol{X})}{|\partial_\rho \boldsymbol{X}|}\frac{\partial_\rho}{|\partial_\rho \boldsymbol{X}|}\nonumber\\
&=-\left(\partial_s\boldsymbol{X}\cdot \partial_s (\partial_t \boldsymbol{X})\right)\partial_s\,.
\end{align}

Applying \eqref{eq: lem2, aux 2} to $\boldsymbol{X}$ and using vector decomposition, we derive that
\begin{equation}\label{eq: lem2, aux 3}
    \partial_t(\partial_s \boldsymbol{X})=\partial_s(\partial_t \boldsymbol{X})-(\partial_s \boldsymbol{X}\cdot \partial_s(\partial_t \boldsymbol{X}))\partial_s \boldsymbol{X} = (\boldsymbol{n}\cdot \partial_s (\partial_t \boldsymbol{X}))\boldsymbol{n}.
\end{equation}
Applying \eqref{eq: lem2, aux 2} to $\partial_s\boldsymbol{X}$, together with \eqref{eq: def of kappa 1} and \eqref{eq: lem2, aux 3}, yields that
\begin{align}\label{eq: lem2, aux 4}
\partial_t (\partial_{ss}\boldsymbol{X})&=\partial_s(\partial_t (\partial_s \boldsymbol{X}))-\left(\partial_s\boldsymbol{X}\cdot \partial_s (\partial_t \boldsymbol{X})\right)\partial_s(\partial_s \boldsymbol{X})\nonumber\\
&=\partial_s((\boldsymbol{n}\cdot \partial_s (\partial_t \boldsymbol{X}))\boldsymbol{n})+\left(\partial_s\boldsymbol{X}\cdot \partial_s (\partial_t \boldsymbol{X})\right)\kappa\boldsymbol{n}\nonumber\\
&=\Big(\partial_s(\boldsymbol{n}\cdot \partial_s (\partial_t \boldsymbol{X}))+\left(\partial_s\boldsymbol{X}\cdot \partial_s (\partial_t \boldsymbol{X})\right)\kappa\Big)\boldsymbol{n}
+ (\boldsymbol{n}\cdot \partial_s (\partial_t \boldsymbol{X})) \partial_s \boldsymbol{n}.
\end{align}

Finally, combining the definition of $\kappa$ in \eqref{eq: def of kappa 1}, identities \eqref{eq: lem2, aux 0} and \eqref{eq: lem2, aux 4}, and the facts $\boldsymbol{n}\cdot \partial_s \boldsymbol{n}=0, \boldsymbol{n}\cdot \partial_t \boldsymbol{n}=0$, we derive $\partial_t \kappa$ as follows:
\begin{align}
\partial_t \kappa&=\boldsymbol{n}\cdot\partial_t (\kappa\, \boldsymbol{n})- \boldsymbol{n}\cdot (\kappa\partial_t \boldsymbol{n})\nonumber\\
&=\boldsymbol{n}\cdot \partial_t(-\partial_{ss}\boldsymbol{X})\nonumber\\
&=-\boldsymbol{n}\cdot \Bigl[\Bigl(\partial_s(\boldsymbol{n}\cdot \partial_s (\partial_t \boldsymbol{X}))+\left(\partial_s\boldsymbol{X}\cdot \partial_s (\partial_t \boldsymbol{X})\right)\kappa\Bigr)\boldsymbol{n}
+ (\boldsymbol{n}\cdot \partial_s (\partial_t \boldsymbol{X})) \partial_s \boldsymbol{n}\Bigr]\nonumber\\
&=-\partial_s(\boldsymbol{n}\cdot \partial_s (\partial_t \boldsymbol{X}))-\left(\partial_s\boldsymbol{X}\cdot \partial_s (\partial_t \boldsymbol{X})\right)\kappa,
\end{align}
which validates \eqref{eq: key identity 2}.
\end{proof}

Let $\boldsymbol{X}(\rho, t)$ be the solution of the geometric PDE of the Willmore flow \eqref{eq: willmore pde origin}, and $V(\rho, t), \kappa(\rho, t)$ are determined via \eqref{eq: willmore pde origin 2} and \eqref{eq: willmore pde origin 3}. We rewrite the first two equations, \eqref{eq: willmore pde origin 1} and \eqref{eq: willmore pde origin 2} into the following differential-algebraic form:
\begin{equation}
    \boldsymbol{n}\cdot \partial_t\boldsymbol{X} = V, \quad V\boldsymbol{n} = (\partial_{ss}\kappa +\frac{1}{2}\kappa^3)\boldsymbol{n}.
\end{equation}
Then by applying Lemma \ref{lem: geo identity for V} and Lemma \ref{lem: geo identity for kappa}, we find that the solution $(\boldsymbol{X}(\rho, t), V(\rho, t), \kappa(\rho, t))$ of the original geometric PDE \eqref{eq: willmore pde origin} also satisfies the following new geometric PDE:
\begin{subequations}
\label{eq: willmore pde new}
\begin{align}
\label{eq: willmore pde new 1}
&\boldsymbol{n}\cdot \partial_t \boldsymbol{X} = V ,\\
\label{eq: willmore pde new 2}
&V\boldsymbol{n} = \partial_s\left(\partial_s \kappa \boldsymbol{n}-\frac{1}{2}\kappa^2\partial_s \boldsymbol{X}\right), \quad \forall \rho \in \mathbb{T}, \, t>0,\\
\label{eq: willmore pde new 3}
&\partial_t\kappa = -\partial_{s}(\boldsymbol{n}\cdot \partial_s\left(\partial_t \boldsymbol{X}\right))-(\partial_s \boldsymbol{X}\cdot \partial_s\left(\partial_t \boldsymbol{X}\right))\, \kappa, 
\end{align}
\end{subequations}
with the initial conditions 
\begin{equation}\label{initpDE3}
\boldsymbol{X}(\rho, 0)=\boldsymbol{X}_0(\rho),\quad
\kappa(\rho, 0)=-\boldsymbol{n}\cdot \partial_{ss}\boldsymbol{X}_0(\rho):=\kappa_0(\rho),
\quad  V(\rho, 0)=\partial_{ss}\kappa(\rho, 0)+\frac{1}{2}\kappa(\rho, 0)^3:=V_0(\rho).
\end{equation}

\begin{remark}
At the continuous level, the time derivative of the curvature, i.e. $\partial_t\kappa$, can be expressed in multiple equivalent formulations. For example, by applying $\partial_t\boldsymbol{X} = V \boldsymbol{n}$ in our formulation \eqref{eq: key identity 2}, and considering the identity $\partial_s \boldsymbol{n}=\kappa \partial_s \boldsymbol{X}$ from \eqref{eq: lem1, aux 2}, we obtain:
\begin{align}\label{eq: huisken kappa}
    \partial_t \kappa &=-\partial_{s}(\boldsymbol{n}\cdot \partial_s\left(V \boldsymbol{n}\right))-(\partial_s \boldsymbol{X}\cdot \partial_s\left(V \boldsymbol{n}\right))\, \kappa\nonumber\\
    &=-\partial_{ss}V-(\partial_s \boldsymbol{X}\cdot \partial_s \boldsymbol{n})\kappa V\nonumber\\
    &=-(\partial_{ss}+\kappa^2) V.
\end{align}
This reveals that our formulation \eqref{eq: key identity 2} is equivalent to the renowned and elegant Huisken's identity for $\partial_t \kappa$ at:  $\partial_t \kappa = -(\partial_{ss}+\kappa^2)V$ \cite{huisken1984flow}. In fact, by replacing our formulation \eqref{eq: key identity 2} with the  Huisken's identity \eqref{eq: huisken kappa} in \eqref{eq: willmore pde new}, we can derive another equivalent but different geometric PDE. Therefore, a geometric flow can be represented by various geometric PDEs.

However, it's crucial to note that although these formulations are equivalent at the continuous level, their resulting numerical schemes can vary significantly at the discretized level. This difference becomes particularly significant in the energy stability analysis --- among all the equivalent formulations we have explored, ours \eqref{eq: willmore pde new} is the only one that successful in proving the energy stability at the full-discretized level. This distinction not only highlights the significance of choosing the appropriate geometric identity  but also emphasizes the importance and potential of our formulation \eqref{eq: key identity 2} and our new geometric PDE \eqref{eq: willmore pde new}.
\end{remark}

\subsection{A new variational formulation and its energy stability}
To introduce the variational formulation for \eqref{eq: willmore pde new}, let's first define the necessary functional spaces. We introduce the space $L^2(\Gamma(t))$ as
\begin{equation}
    L^2(\Gamma(t)):=\Bigl\{u: \Gamma(t)\to \mathbb{R}\, \Big| \int_{\Gamma(t)} |u|^2 ds<\infty\Bigr\},
\end{equation}
equipped with the weighted inner product $(\cdot, \cdot)_{\Gamma(t)}$ defined by
\begin{equation}
    \left(u, v\right)_{\Gamma(t)}:=\int_{\Gamma(t)}u\, v\, ds, \qquad \forall u, v \in L^2(\Gamma(t)).
\end{equation}
This definition is naturally extended to $[L^2(\Gamma(t))]^2$. Moreover, we define the Sobolev space $H^1(\Gamma(t))$ as
\begin{equation}
    H^1(\Gamma(t)):=\Bigl\{u: \Gamma(t)\to \mathbb{R}\, \Big| u \in L^2(\Gamma(t)), \partial_s u\in L^2(\Gamma(t))\Bigr\}.
\end{equation}
The extension to $[H^1(\Gamma(t))]^2$ follows straightforwardly.

We multiply test functions $\phi \in L^2(\Gamma(t))$, $\boldsymbol{\omega} \in [H^1(\Gamma(t))]^2$, and $\psi \in H^1(\Gamma(t))$ in equations \eqref{eq: willmore pde new 1}-\eqref{eq: willmore pde new 3}, respectively,  integrate over $\Gamma(t)$, and apply integration by parts, respectively. This leads to the following variational formulation of \eqref{eq: willmore pde new}: Given the initial curve $\Gamma_0$ parameterized by $\boldsymbol{X}_0$, i.e. $\boldsymbol{X}(\cdot, 0)=\boldsymbol{X}_0$, $\kappa(\rho,0)=\kappa_0(\rho)$ and $V(\rho,0)=V_0(\rho)$ in 
\eqref{initpDE3}, for any $t\geq 0$, find the solution  $(\boldsymbol{X}(\cdot, t), V(\cdot, t), \kappa(\cdot, t))\in [H^1(\Gamma(t))]^2\times L^2(\Gamma(t))\times H^1(\Gamma(t))$ such that
\begin{subequations}
\label{eq: willmore weak}
\begin{align}
\label{eq: willmore weak 1}
&\Bigl(\boldsymbol{n}\cdot \partial_t\boldsymbol{X},\, \phi\Bigr)_{\Gamma(t)} = \Bigl(V,\,\phi\Bigr)_{\Gamma(t)}, \quad \forall \phi \in L^2(\Gamma(t)),\\[0.5em]
\label{eq: willmore weak 2}
&\Bigl(V\boldsymbol{n},\,\boldsymbol{\omega}\Bigr)_{\Gamma(t)}=\Bigl(-\partial_s \kappa\, \boldsymbol{n}+\frac{1}{2}\kappa^2\partial_s \boldsymbol{X},\, \partial_s\boldsymbol{\omega}\Bigr)_{\Gamma(t)}, \quad\forall \boldsymbol{\omega} \in [H^1(\Gamma(t))]^2,\\[0.5em]
\label{eq: willmore weak 3}
&\Bigl(\partial_t \kappa,\,\psi\Bigr)_{\Gamma(t)} = \Bigl(\boldsymbol{n}\cdot \partial_s\left(\partial_t \boldsymbol{X}\right),\, \partial_s \psi\Bigr)_{\Gamma(t)}
-\Bigl((\partial_s \boldsymbol{X}\cdot \partial_s\left(\partial_t \boldsymbol{X}\right))\, \kappa,\psi\Bigr)_{\Gamma(t)}, \quad \forall \psi \in H^1(\Gamma(t)),
\end{align}
where $\boldsymbol{n}=-(\partial_s \boldsymbol{X})^\perp$. 
\end{subequations}

For the new variational formulation \eqref{eq: willmore weak}, we have
\begin{theorem}[Energy dissipation]\label{thm: energy, continuous}Let $(\boldsymbol{X}(\cdot, t), V(\cdot, t), \kappa(\cdot, t))\in [H^1(\Gamma(t))]^2\times L^2(\Gamma(t))\times H^1(\Gamma(t))$ be a solution of the variational problem \eqref{eq: willmore weak} with the initial curve $\Gamma_0$. Then the Willmore energy of $\Gamma(t)$, as defined in \eqref{eq: def of willmore energy, continuous}, is decreasing over time, i.e.,
\begin{equation}\label{eq: energy dissipation, weak}
    W(t)\leq W(t')\leq W(0)=\int_{\Gamma_0}\kappa^2 \, ds, \quad \forall t\geq t'\geq 0.
\end{equation}
\end{theorem}
\begin{proof}
We start by differentiating $W(t)$ with respect to time $t$. Recalling \eqref{eq: lem2, aux 1}, we can formulate this derivative as follows
\begin{align}\label{eq: thm1, aux 1}
\frac{d}{dt}W(t)&=\frac{d}{dt}\int_{\Gamma(t)} \frac{\kappa^2}{2}\, ds\nonumber\\
&=\frac{d}{dt}\int_0^1 \frac{\kappa^2}{2} |\partial_\rho \boldsymbol{X}|\, d\rho\nonumber\\
&=\int_0^1 \kappa \,\partial_t\kappa \, |\partial_\rho \boldsymbol{X}|\,d\rho+\int_0^1 \frac{\kappa^2}{2} \partial_t|\partial_\rho \boldsymbol{X}| \,d\rho\nonumber\\
&=\int_{\Gamma(t)} \kappa \,\partial_t\kappa \,ds+\int_0^1 \frac{\kappa^2}{2} |\partial_\rho \boldsymbol{X}|\,\partial_s\boldsymbol{X}\cdot \partial_s (\partial_t \boldsymbol{X}) \,d\rho\nonumber\\
&=\int_{\Gamma(t)} \kappa \,\partial_t\kappa \,ds+\int_{\Gamma(t)} \frac{\kappa^2}{2} \partial_s \partial_t\boldsymbol{X}\cdot \partial_s\boldsymbol{X}  \,ds\nonumber\\
&=\Bigl(\partial_t\kappa, \kappa\Bigr)_{\Gamma(t)}+\Bigl(\frac{\kappa^2}{2}\partial_s \boldsymbol{X}, \partial_s(\partial_t\boldsymbol{X})\Bigr)_{\Gamma(t)}.
\end{align}
On the other hand, by taking test functions $\phi = -V$ in \eqref{eq: willmore weak 1}, $\boldsymbol{\omega}=\partial_t \boldsymbol{X}$ in \eqref{eq: willmore weak 2}, and $\psi = \kappa$ in \eqref{eq: willmore weak 3}, and combining these equations, we obtain
\begin{align}\label{eq: thm1, aux 2}
\Bigl(\partial_t\kappa, \kappa\Bigr)_{\Gamma(t)}=-\Bigl(V, V\Bigr)_{\Gamma(t)}-\Bigl(\frac{\kappa^2}{2}\partial_s \boldsymbol{X}, \partial_s(\partial_t\boldsymbol{X})\Bigr)_{\Gamma(t)}.
\end{align}
Combining \eqref{eq: thm1, aux 1} and \eqref{eq: thm1, aux 2}, we deduce that
\begin{equation}
    \frac{d}{dt}W(t)=-\Bigl(V, V\Bigr)_{\Gamma(t)}\leq 0, \qquad t\ge0,
\end{equation}
which implies \eqref{eq: energy dissipation, weak}. 
\end{proof}

\section{An energy-stable semi-discretization}
\subsection{A spatial semi-discretization}
We first address the discretization of the domain $\mathbb{T}=[0, 1]$. The periodic unit interval $\mathbb{T}=[0, 1]$ is uniformly partitioned into $N\geq 3$ subintervals $I_j = [\rho_{j-1}, \rho_j]$ with nodes at $\rho_j = j h$ for $j = 1, \ldots, N$, where $h=\frac{1}{N}$ represents the mesh size. By the periodicity property of $\mathbb{T}$, we define $0=\rho_0=\rho_N$. We thus introduce the finite element space, which consists of continuous piecewise linear functions in $\cup_{j=1}^N I_j$ as
\begin{equation}
    \mathbb{K}^h:=\Bigl\{u^h\in C(\mathbb{T})\, \Big| u^h|_{I_j} \text{ is linear }\forall j = 1,  \ldots, N, u^h(\rho_0)=u^h(\rho_N) \Bigr\}.
\end{equation}

To discretize the evolving planar curve $\Gamma$, we consider its parameterization $\boldsymbol{X}$, which can be approximated as $\boldsymbol{X}^h(\mathbb{T}, t)=(x^h(\mathbb{T}, t), y^h(\mathbb{T}, t))^T\in [\mathbb{K}^h]^2$. This $\boldsymbol{X}^h$ leads to a polygonal curve $\Gamma^h$, serving as an approximation of $\Gamma$. The polygonal curve $\Gamma^h$ is composed of $N$ ordered line segments $\boldsymbol{h}_j$, which can also be formulated as:
\begin{equation}\label{eq: def of h}
    \Gamma^h=\bigcup\limits_{j=1}^N \boldsymbol{h}_j, \quad \boldsymbol{h}_j:=\boldsymbol{X}^{h}(\rho_{j})-\boldsymbol{X}^h(\rho_{j-1}),
\end{equation}
where $\boldsymbol{h}_j\in \mathbb{R}^2$ is a line segment with two vertices $\boldsymbol{X}^h(\rho_{j-1}), \boldsymbol{X}^h(\rho_{j})$ for $j = 1, \ldots, N$. By using $\boldsymbol{h}_j$, the discretized arc length derivative $\partial_s$ for a function $u^h\in \mathbb{K}^h$ is then defined as
\begin{equation}\label{eq: discrete arc length derivative}
    \partial_s u^h|_{I_j}:=\frac{u^h(\rho_j)-u^h(\rho_{j-1})}{|\boldsymbol{h}_j|}, \quad \forall j = 1, \ldots, N.
\end{equation}
Similarly, the discretized unit tangent and outward normal vectors, $\boldsymbol{\tau}^h$ and $\boldsymbol{n}^h$, are determined according to \eqref{eq: discrete tau and n} and \eqref{eq: discrete arc length derivative} as
\begin{equation}\label{eq: discrete tau and n}
    \boldsymbol{\tau}_j^h=\boldsymbol{\tau}^h|_{I_j}:=\partial_s \boldsymbol{X}^h|_{I_j}=\frac{\boldsymbol{h}_j}{|\boldsymbol{h}_j|}, \quad \boldsymbol{n}_j^h = \boldsymbol{n}^h|_{I_j}:=-(\boldsymbol{\tau}_j^h)^\perp = -\frac{\boldsymbol{h}_j^\perp}{|\boldsymbol{h}_j|}.
\end{equation}

Finally, for the scalar-/vector-valued functions $u, v\in \mathbb{K}^h/[\mathbb{K}^h]^2$, we adopt the following mass-lumped inner product $\left(\cdot, \cdot\right)_{\Gamma^h}^h$ with respect to $\Gamma^h$ to approximate the weighted inner product $\left(\cdot, \cdot\right)_{\Gamma}$ with respect to $\Gamma$ as
\begin{equation}
    \left(u^h, v^h\right)_{\Gamma^h}^h:=\frac{1}{2}\sum_{j=1}^N |\boldsymbol{h}_j|\Bigl[(u^h\cdot v^h)(\rho_{j-1}^+)+(u^h\cdot v^h)(\rho_j^-)\Bigr], \quad \forall u^h, v^h\in \mathbb{K}^h/ [\mathbb{K}^h]^2,
\end{equation}
where $u(\rho_j^{\pm}):=\lim_{\rho\to \rho_j^{\pm}} u(\rho), \forall j = 0, \ldots, N$. 

Based on the above definitions, we now propose an energy-stable spatial semi-discretization for the variational formulation \eqref{eq: willmore weak}: Given the initial polygonal curve $\Gamma^h_0:=\boldsymbol{X}^h_0$ and the initial approximations $\boldsymbol{X}^h(\cdot, 0)=\boldsymbol{X}_0^h(\cdot)\in [\mathbb{K}^h]^2$, 
$V^h(\cdot, 0)=V_0^h(\cdot)\in \mathbb{K}^h$ and $\kappa^h(\cdot, 0)=\kappa_0^h(\cdot)\in \mathbb{K}^h$
satisfying  $\boldsymbol{X}_0^h(\rho_j)=\boldsymbol{X}_0(\rho_j), V_0^h(\rho_j)=V_0(\rho_j), \kappa_0^h(\rho_j)=\kappa_0(\rho_j)$ for $j = 0, \ldots, N$. For any $t\geq 0$, find the solution $\Gamma^h(t):=(\boldsymbol{X}^h(\cdot, t), V^h(\cdot, t), \kappa^h(\cdot, t))\in [\mathbb{K}^h]^2\times \mathbb{K}^h\times \mathbb{K}^h$, such that 
\begin{subequations}
\label{eq: willmore semi}
\begin{align}
\label{eq: willmore semi 1}
&\Bigl(\boldsymbol{n}^h\cdot \partial_t\boldsymbol{X}^h,\, \phi^h\Bigr)_{\Gamma^h(t)}^h = \Bigl(V^h,\,\phi^h\Bigr)_{\Gamma^h(t)}^h, \quad \forall \phi^h \in \mathbb{K}^h,\\[0.5em]
\label{eq: willmore semi 2}
&\Bigl(V^h\boldsymbol{n}^h,\,\boldsymbol{\omega}^h\Bigr)_{\Gamma^h(t)}^h=\Bigl(-\partial_s \kappa^h \boldsymbol{n}^h+\frac{1}{2}(\kappa^h)^2\partial_s \boldsymbol{X}^h,\, \partial_s\boldsymbol{\omega^h}\Bigr)_{\Gamma^h(t)}^h, \;\forall \boldsymbol{\omega}^h \in [\mathbb{K}^h]^2,\\[0.5em]
\label{eq: willmore semi 3}
&\Bigl(\partial_t \kappa^h,\psi^h\Bigr)_{\Gamma^h(t)}^h = \Bigl(\boldsymbol{n}^h\cdot \partial_s(\partial_t \boldsymbol{X}^h),\partial_s \psi^h\Bigr)_{\Gamma^h(t)}^h
-\Bigl((\partial_s \boldsymbol{X}^h\cdot \partial_s(\partial_t \boldsymbol{X}^h)) 
\kappa^h,\psi^h\Bigr)_{\Gamma^h(t)}^h, \ \forall \psi^h \in \mathbb{K}^h.
\end{align}
\end{subequations}

\subsection{Energy stability}
Define the discretized Willmore energy $W^h$ of $\Gamma^h$ as
\begin{equation}\label{eq: def of willmore energy, semi}
    W^h(t):=W^h(\Gamma^h(t)):=\frac{1}{2}\Bigl((\kappa^h(\cdot,t)^2, 1\Bigr)_{\Gamma^h}^h = \frac{1}{4}\sum_{j=1}^N |\boldsymbol{h}_j|\Bigl[(\kappa^h(\rho_{j-1}^+,t))^2+(\kappa^h(\rho_j^-,t))^2\Bigr].
\end{equation}
We show that the semi-discretization \eqref{eq: willmore semi} is a conformal discretization of \eqref{eq: willmore weak}, i.e., its solution is energy-stable.
\begin{theorem}[Energy dissipation]\label{thm: energy, semi}Let $(\boldsymbol{X}^h(\cdot, t), V^h(\cdot, t), \kappa^h(\cdot, t))\in [\mathbb{K}^h]^2\times \mathbb{K}^h\times \mathbb{K}^h$ be a solution of the semi-discritization \eqref{eq: willmore semi} with the initial curve $\Gamma_0^h$. Then the discretized Willmore energy of $\Gamma^h(t)$ given in \eqref{eq: def of willmore energy, semi} is decreasing over time, i.e.,
\begin{equation}\label{eq: energy dissipation, semi}
    W^h(t)\leq W^h(t')\leq W^h(0)=\frac{1}{4}\sum_{j=1}^N |\boldsymbol{h}_{0_j}|\Bigl[((\kappa^h_0)^2)(\rho_{j-1}^+)+((\kappa^h_0)^2)(\rho_j^-)\Bigr], \quad \forall t\geq t'\geq 0.
\end{equation}

\end{theorem}
\begin{proof}To prove this theorem, we start with the temporal derivative of the length of each segment $\boldsymbol{h}_j$. From \eqref{eq: def of h}, we deduce that
\begin{align}\label{eq: thm2, aux 1}
    \partial_t |\boldsymbol{h}_j|&=\frac{\partial_t(x^h(\rho_{j}, t)-x^h(\rho_{j-1}, t))(x^h(\rho_{j}, t)-x^h(\rho_{j-1}, t))}{\sqrt{(x^h(\rho_j, t)-x^h(\rho_{j-1}, t))^2+(y^h(\rho_j, t)-y^h(\rho_{j-1}, t))^2}}\nonumber\\
    &\quad + \frac{\partial_t(y^h(\rho_{j}, t)-y^h(\rho_{j-1}, t))(y^h(\rho_{j}, t)-y^h(\rho_{j-1}, t))}{\sqrt{(x^h(\rho_j, t)-x^h(\rho_{j-1}, t))^2+(y^h(\rho_j, t)-y^h(\rho_{j-1}, t))^2}}\nonumber\\
    &=|\boldsymbol{h}_j|\left(\partial_s \boldsymbol{X}^h\cdot \partial_s(\partial_t\boldsymbol{X}^h)\right)
    \Big|_{I_j}.
\end{align}
Considering the derivative of $W^h(t)$ with respect to $t$ and employing \eqref{eq: thm2, aux 1}, we find
\begin{align}\label{eq: thm2, aux 2}
\frac{d}{dt}W^h(t)&=\frac{d}{dt}\frac{1}{4}\sum_{j=1}^N |\boldsymbol{h}_j|\Bigl[(\kappa^h(\rho_{j-1}^+,t))^2+(\kappa^h(\rho_j^-),t)^2\Bigr]\nonumber\\
&=\frac{1}{4}\sum_{j=1}^N |\boldsymbol{h}_j|\left(\partial_s \boldsymbol{X}^h\cdot \partial_s(\partial_t\boldsymbol{X}^h)\right)\Big|_{I_j}\Bigl[((\kappa^h)^2)(\rho_{j-1}^+)+((\kappa^h)^2)(\rho_j^-)\Bigr]\nonumber\\
&\qquad+\frac{1}{2}\sum_{j=1}^N |\boldsymbol{h}_j|\Bigl[(\kappa^h\partial_t \kappa^h)(\rho_{j-1}^+)+(\kappa^h\partial_t \kappa^h)(\rho_j^-)\Bigr]\nonumber\\
&=\Bigl(\frac{(\kappa^h)^2}{2}\partial_s \boldsymbol{X}^h, \partial_s(\partial_t\boldsymbol{X}^h)\Bigr)_{\Gamma^h(t)}^h+\Bigl(\partial_t\kappa^h, \kappa^h\Bigr)_{\Gamma^h(t)}^h.
\end{align}
Following a similar approach as in the proof of Theorem \ref{thm: energy, continuous}, we select $\phi^h = -V^h$ in \eqref{eq: willmore semi 1}, $\boldsymbol{\omega}^h=\partial_t \boldsymbol{X}^h$ in \eqref{eq: willmore semi 2}, and $\psi^h = \kappa^h$ in \eqref{eq: willmore semi 3}. Utilizing \eqref{eq: thm2, aux 2}, we derive that
\begin{equation}
    \frac{d}{dt}W^h(t)=-\Bigl(V^h, V^h\Bigr)_{\Gamma^h(t)}^h\leq 0,
\end{equation}
which indicates the energy dissipation \eqref{eq: energy dissipation, semi}. 
\end{proof}

\section{An energy-stable PFEM}
\subsection{A full-discretization}
To derive a full-discertization of the variational formulation \eqref{eq: willmore weak}, we define $\tau>0$ as the uniform time step and denote the discrete time levels by $t_m:=m\tau$ for $m\ge0$. Let $\Gamma^m:=
(\boldsymbol{X}^m,V^m, \kappa^m)\in [\mathbb{K}^h]^2\times \mathbb{K}^h\times 
\mathbb{K}^h$ represents the approximation of $\Gamma^h(t_m):=(\boldsymbol{X}^h(\cdot, t_m), V^h(\cdot, t_m), \kappa^h(\cdot, t_m))\in [\mathbb{K}^h]^2\times \mathbb{K}^h\times \mathbb{K}^h$. We know that the curve $\Gamma^m$ is also a polygonal curve, and is composed of $N$ ordered line segments $\boldsymbol{h}^m_j$ as
\begin{equation}
    \boldsymbol{h}_j^m:=\boldsymbol{X}^m(\rho_{j})-\boldsymbol{X}^m(\rho_{j-1}), \quad \forall j=1, \ldots, N.
\end{equation}

The discretization of the arc length derivative $\partial_s$, the unit tangent vector $\boldsymbol{\tau}^m$, and the outward unit normal vector $\boldsymbol{n}^m$ for $\Gamma^m$ follows a similar argument to \eqref{eq: discrete arc length derivative} and \eqref{eq: discrete tau and n}. Consequently, the mass-lumped inner product on $\Gamma^m$ is defined as
\begin{equation}
    \left(u^h, v^h\right)_{\Gamma^m}:=\frac{1}{2}\sum_{j=1}^N |\boldsymbol{h}_j^m|\Bigl[(u^h\cdot v^h)(\rho_{j-1}^+)+(u^h\cdot v^h)(\rho_j^-)\Bigr], \quad \forall u^h, v^h\in \mathbb{K}^h/ [\mathbb{K}^h]^2.
\end{equation}

By adopting the backward Euler in time, we propose an energy-stable full-discritization of the variational formulation \eqref{eq: willmore weak} as follows: Given the initial polygonal curve $\Gamma^0:=\Gamma_0^h=\boldsymbol{X}^h_0$ and the initial approximations $(\boldsymbol{X}^0, V^0, \kappa^0)=(\boldsymbol{X}_0^h, V_0^h, \kappa_0^h)$. For any $m\geq 0$, find the solution $\Gamma^{m+1}:=(\boldsymbol{X}^{m+1}, V^{m+1}, \kappa^{m+1})\in [\mathbb{K}^h]^2\times \mathbb{K}^h\times \mathbb{K}^h$, such that 
\begin{subequations}
\label{eq: willmore full}
\begin{align}
\label{eq: willmore full 1}
&\Bigl(\boldsymbol{n}^{m}\cdot \frac{\boldsymbol{X}^{m+1}-\boldsymbol{X}^m}{\tau},\, \phi^h\Bigr)_{\Gamma^m} = \Bigl(V^{m+1},\,\phi^h\Bigr)_{\Gamma^m}, \quad \forall \phi^h \in \mathbb{K}^h,\\[0.5em]
\label{eq: willmore full 2}
&\Bigl(V^{m+1}\boldsymbol{n}^m,\,\boldsymbol{\omega}^h\Bigr)_{\Gamma^m}=\Bigl(-\partial_s \kappa^{m+1} \boldsymbol{n}^m+\frac{1}{2}(\kappa^{m+1})^2\partial_s \boldsymbol{X}^{m+1},\, \partial_s\boldsymbol{\omega^h}\Bigr)_{\Gamma^m}, \quad \forall \boldsymbol{\omega}^h \in [\mathbb{K}^h]^2,\\[0.5em]
\label{eq: willmore full 3}
&\Bigl(\delta\kappa^m,\psi^h\Bigr)_{\Gamma^m} = \Bigl(\boldsymbol{n}^m\cdot \partial_s(\delta \boldsymbol{X}^m), \partial_s \psi^h\Bigr)_{\Gamma^m}
-\Bigl(\partial_s \boldsymbol{X}^{m+1}\cdot \partial_s(\delta \boldsymbol{X}^m) \kappa^{m+1},\psi^h\Bigr)_{\Gamma^m}, \ \forall\psi^h \in \mathbb{K}^h,
\end{align}
\end{subequations}
where $\delta\kappa^m:=\frac{\kappa^{m+1}-\kappa^m}{\tau}$ and $\delta \boldsymbol{X}^m:=\frac{\boldsymbol{X}^{m+1}-\boldsymbol{X}^m}{\tau}$.

\subsection{Energy stability}
Denote the discretized Willmore energy of $\Gamma^m$ as $W^m:=W^h(\Gamma^m)$. Then for the full-discretization of the Willmore flow \eqref{eq: willmore full}, we have the following unconditional energy stability result.

\smallskip 
\begin{theorem}[Unconditional energy dissipation]\label{thm: energy, full}Let $(\boldsymbol{X}^m, V^m, \kappa^m)\in [\mathbb{K}^h]^2\times \mathbb{K}^h\times \mathbb{K}^h$ be a solution of the energy-stable PFEM \eqref{eq: willmore full} with the initial curve $\Gamma_0^h$. Then the discretized Willmore energy $W^m$ is unconditional decreasing with respect to $m$, i.e.,
\begin{equation}\label{eq: energy dissipation, full}
    W^{m+1}\leq W^m\leq W^0=\frac{1}{4}\sum_{j=1}^N |\boldsymbol{h}_j^0|\Bigl[((\kappa^0)^2)(\rho_{j-1}^+)+((\kappa^0)^2)(\rho_j^-)\Bigr], \quad \forall m\geq 0.
\end{equation}
\end{theorem}
\begin{proof}Taking $\phi^h=-\tau V^{m+1}$ in \eqref{eq: willmore full 1}, $\boldsymbol{\omega}^h=\boldsymbol{X}^{m+1}-\boldsymbol{X}^m$ in \eqref{eq: willmore full 2}, and $\psi^h = \tau\kappa^{m+1}$ in \eqref{eq: willmore full 3}, we derive the following identities:
\begin{subequations}
\label{eq: aux identities}
\begin{align}
\label{eq: aux identity 1}
-\Bigl(\boldsymbol{n}^{m}\cdot (\boldsymbol{X}^{m+1}-\boldsymbol{X}^m),\, V^{m+1}\Bigr)_{\Gamma^m}
=&-\tau\Bigl(V^{m+1},\,V^{m+1}\Bigr)_{\Gamma^m}, \\[0.5em]
\label{eq: aux identity 2}
\Bigl(V^{m+1}\boldsymbol{n}^m,\,\boldsymbol{X}^{m+1}-\boldsymbol{X}^m\Bigr)_{\Gamma^m}=&\Bigl(-\partial_s \kappa^{m+1} \boldsymbol{n}^m,\, \partial_s(\boldsymbol{X}^{m+1}-\boldsymbol{X}^m)\Bigr)_{\Gamma^m}\nonumber\\
&+\frac{1}{2}\Bigl((\kappa^{m+1})^2\partial_s \boldsymbol{X}^{m+1},\, \partial_s(\boldsymbol{X}^{m+1}-\boldsymbol{X}^m)\Bigr)_{\Gamma^m},\\[0.5em]
\label{eq: aux identity 3}
\Bigl(\kappa^{m+1}-\kappa^m,\,\kappa^{m+1}\Bigr)_{\Gamma^m} =&\Bigl(\boldsymbol{n}^m\cdot \partial_s\left(\boldsymbol{X}^{m+1}-\boldsymbol{X}^m\right),\, \partial_s \kappa^{m+1}\Bigr)_{\Gamma^m}\nonumber\\
&-\Bigl(\partial_s \boldsymbol{X}^{m+1}\cdot \partial_s\left(\boldsymbol{X}^{m+1}-\boldsymbol{X}^m\right)\, \kappa^{m+1},\kappa^{m+1}\Bigr)_{\Gamma^m}.
\end{align}
\end{subequations}
Summing these three equations, we obtain that
\begin{align}\label{eq: aux identity 4}
    \Bigl(\kappa^{m+1}-\kappa^m, \,\kappa^{m+1}\Bigr)_{\Gamma^m}=&-\tau\Bigl(V^{m+1},\, V^{m+1}\Bigr)_{\Gamma^m}\nonumber\\
    &-\frac{1}{2}\Bigl(\left(\kappa^{m+1}\right)^2\partial_s \boldsymbol{X}^{m+1},\, \partial_s(\boldsymbol{X}^{m+1}-\boldsymbol{X}^m)\Bigr)_{\Gamma^m}\,.
\end{align}
For $W^{m+1}$, we have the inequality:
\begin{align}\label{Wmp1456}
&\frac{1}{2}\Bigl(\kappa^{m+1}, \kappa^{m+1}\Bigr)_{\Gamma^m}+\frac{1}{4}\Bigl(\left(\kappa^{m+1}\right)^2\partial_s \boldsymbol{X}^{m+1}, \partial_s \boldsymbol{X}^{m+1}\Bigr)_{\Gamma^m}-\frac{1}{4}\Bigl(\left(\kappa^{m+1}\right)^2\partial_s \boldsymbol{X}^{m}, \partial_s \boldsymbol{X}^{m}\Bigr)_{\Gamma^m}\nonumber\\
&=\frac{1}{4}\sum_{j=1}^N |\boldsymbol{h}_j^m|\Bigl[(\kappa^{m+1})^2(\rho_{j-1})+(\kappa^{m+1})^2(\rho_j)\Bigr]\nonumber\\
&\quad + \frac{1}{8}\sum_{j=1}^N |\boldsymbol{h}_j^m|\Bigl[(\kappa^{m+1})^2(\rho_{j-1})\frac{\boldsymbol{h}_j^{m+1}\cdot \boldsymbol{h}_j^{m+1}}{|\boldsymbol{h}_j^m|^2}+(\kappa^{m+1})^2(\rho_j)\frac{\boldsymbol{h}_j^{m+1}\cdot \boldsymbol{h}_j^{m+1}}{|\boldsymbol{h}_j^m|^2}\Bigr]\nonumber\\
&\quad - \frac{1}{8}\sum_{j=1}^N |\boldsymbol{h}_j^m|\Bigl[(\kappa^{m+1})^2(\rho_{j-1})\frac{\boldsymbol{h}_j^{m}\cdot \boldsymbol{h}_j^{m}}{|\boldsymbol{h}_j^m|^2}+(\kappa^{m+1})^2(\rho_j)\frac{\boldsymbol{h}_j^{m}\cdot \boldsymbol{h}_j^{m}}{|\boldsymbol{h}_j^m|^2}\Bigr]\nonumber\\
&=\sum_{j=1}^N \frac{|\boldsymbol{h}_j^m|+\frac{|\boldsymbol{h}_j^{m+1}|^2}{|\boldsymbol{h}_j^m|}}{8}\Bigl[(\kappa^{m+1})^2(\rho_{j-1})+(\kappa^{m+1})^2(\rho_j)\Bigr]\nonumber\\
&\geq \sum_{j=1}^N \frac{|\boldsymbol{h}_j^{m+1}|}{4}\Bigl[(\kappa^{m+1})^2(\rho_{j-1})+(\kappa^{m+1})^2(\rho_j)\Bigr]\nonumber\\
&=W^{m+1}.
\end{align}
Applying the inequality $a(a-b)\geq \frac{1}{2}a^2-\frac{1}{2}b^2$, we deduce that
\begin{align}
\label{eq: aux identity 5}
\Bigl(\kappa^{m+1}-\kappa^m, \,\kappa^{m+1}\Bigr)_{\Gamma^m}\geq&\frac{1}{2}\Bigl(\kappa^{m+1}, \kappa^{m+1}\Bigr)_{\Gamma^m}-\frac{1}{2}\Bigl(\kappa^{m}, \kappa^{m}\Bigr)_{\Gamma^m}\nonumber\\
=&\frac{1}{2}\Bigl(\kappa^{m+1}, \kappa^{m+1}\Bigr)_{\Gamma^m}-W^m,
\end{align}
\begin{align}
\label{eq: aux identity 6}
    &\frac{1}{2}\Bigl(\left(\kappa^{m+1}\right)^2\partial_s \boldsymbol{X}^{m+1},\, \partial_s(\boldsymbol{X}^{m+1}-\boldsymbol{X}^m)\Bigr)_{\Gamma^m}\nonumber\\
    &\ \geq \frac{1}{4}\Bigl(\left(\kappa^{m+1}\right)^2\partial_s \boldsymbol{X}^{m+1}, \partial_s \boldsymbol{X}^{m+1}\Bigr)_{\Gamma^m}
    -\frac{1}{4}\Bigl(\left(\kappa^{m+1}\right)^2\partial_s \boldsymbol{X}^{m}, \partial_s \boldsymbol{X}^{m}\Bigr)_{\Gamma^m}.
\end{align}
Combining \eqref{eq: aux identity 4}, \eqref{Wmp1456},  \eqref{eq: aux identity 5} and \eqref{eq: aux identity 6}, we conclude that
\begin{align*}
    0&\geq-\tau\Bigl(V^{m+1},\, V^{m+1}\Bigr)_{\Gamma^m}\\
    &=\Bigl(\kappa^{m+1}-\kappa^m, \,\kappa^{m+1}\Bigr)_{\Gamma^m}+\Bigl(\frac{1}{2}\left(\kappa^{m+1}\right)^2\partial_s \boldsymbol{X}^{m+1},\, \partial_s(\boldsymbol{X}^{m+1}-\boldsymbol{X}^m)\Bigr)_{\Gamma^m}\\
    &\geq W^{m+1}-W^m.
\end{align*}
Which is the desired energy stability \eqref{eq: energy dissipation, full}.
\end{proof}
\subsection{An iterative solver}
The energy-stable PFEM \eqref{thm: energy, full} is fully-implicit. To solve for $(\boldsymbol{X}^{m+1}, V^{m+1}, \kappa^{m+1})$ in  $[\mathbb{K}^h]^2\times \mathbb{K}^h\times \mathbb{K}^h$, we employ the following Newton-Raphson iteration: 

Set the initial approximation $(\boldsymbol{X}^{m+1, 0}, V^{m+1, 0}, \kappa^{m+1, 0})$ to be the solution from the previous time step $(\boldsymbol{X}^m, V^m, \kappa^m)$. For each iteration $i$, the next approximation $(\boldsymbol{X}^{m+1, i+1}, V^{m+1, i+1}, \kappa^{m+1, i+1})$ is determined by the current approximation $(\boldsymbol{X}^{m+1, i}, V^{m+1, i}, \kappa^{m+1, i})$ as follows
\begin{subequations}
\label{eq: Newton}
\begin{align}
\label{eq: Newton 1}
&\Bigl(\boldsymbol{n}^m\cdot\frac{\boldsymbol{X}^{m+1, i+1}-\boldsymbol{X}^m}{\tau}, \phi^h\Bigr)_{\Gamma^m}=\Bigl(V^{m+1, i+1},\, \phi^h\Bigr)_{\Gamma^m},\quad \forall \phi^h \in \mathbb{K}^h,\\[1em]
\label{eq: Newton 2}
&\Bigl(V^{m+1, i+1}\boldsymbol{n}^m,\boldsymbol{\omega}^h\Bigr)_{\Gamma^m}=\Bigl(\partial_s \kappa^{m+1, i+1} \boldsymbol{n}^{m}+\frac{1}{2}\left(\kappa^{m+1, i}\right)^2\partial_s \boldsymbol{X}^{m+1, i+1},\, \partial_s\boldsymbol{\omega}^h\Bigr)_{\Gamma^m}\nonumber\\
&\quad\quad\quad\quad +\Bigl((\kappa^{m+1, i+1}-\kappa^{m+1, i})\kappa^{m+1, i}\partial_s \boldsymbol{X}^{m+1, i},\, \partial_s\boldsymbol{\omega}^h\Bigr)_{\Gamma^m}, \quad \forall \boldsymbol{\omega}^h \in [\mathbb{K}^h]^2, \\[1em]
\label{eq: Newton 3}
&\Bigl(\kappa^{m+1, i+1}-\kappa^m, \psi^h\Bigr)_{\Gamma^m} = \Bigl(\boldsymbol{n}^m\cdot \partial_s\left(\boldsymbol{X}^{m+1, i+1}-\boldsymbol{X}^m\right),\, \partial_s \psi^h\Bigr)_{\Gamma^m} \nonumber\\
& \quad\quad\qquad   -\Bigl(\partial_s \boldsymbol{X}^{m+1, i+1}\cdot \partial_s\left(\boldsymbol{X}^{m+1, i}-\boldsymbol{X}^m\right)\, \kappa^{m+1, i},\psi^h\Bigr)_{\Gamma^m}\nonumber\\
&\quad\quad\qquad    -\Bigl(\partial_s \boldsymbol{X}^{m+1, i}\cdot \partial_s\left(\boldsymbol{X}^{m+1, i+1}-\boldsymbol{X}^{m+1, i}\right)\, \kappa^{m+1, i},\psi^h\Bigr)_{\Gamma^m}\nonumber\\
&\quad\quad\qquad   -\Bigl(\partial_s \boldsymbol{X}^{m+1, i}\cdot \partial_s\left(\boldsymbol{X}^{m+1, i}-\boldsymbol{X}^m\right)\, (\kappa^{m+1, i+1}-\kappa^{m+1, i}),\psi^h\Bigr)_{\Gamma^m},\quad \forall \psi^h \in \mathbb{K}^h.
\end{align}
\end{subequations}
The iterative process is repeated until all of the following three convergence criteria are satisfied:
\begin{subequations}
\begin{align*}
&\left\|\boldsymbol{X}^{m+1, i+1}-\boldsymbol{X}^{m+1, i}\right\|_{L^\infty}=\max_{1\leq j\leq N} |\boldsymbol{X}^{m+1, i+1}(\rho_j)-\boldsymbol{X}^{m+1, i}(\rho_j)|\leq \text{tol}.\\
&\left\|V^{m+1, i+1}-V^{m+1, i}\right\|_{L^\infty}=\max_{1\leq j\leq N} |V^{m+1, i+1}(\rho_j)-V^{m+1, i}(\rho_j)|\leq \text{tol}.\\
&\left\|\kappa^{m+1, i+1}-\kappa^{m+1, i}\right\|_{L^\infty}=\max_{1\leq j\leq N} |\kappa^{m+1, i+1}(\rho_j)-\kappa^{m+1, i}(\rho_j)|\leq \text{tol}.
\end{align*}
\end{subequations}
Here tol is a predefined tolerance level, ensuring the iterative process continues until the solution achieves a certain accuracy.

\section{Numerical results}
In this section, we present several numerical tests to illustrate the performance of the energy-stable PFEM \eqref{eq: willmore full}. 

We note that our numerical scheme \eqref{eq: willmore full} requires the initial condition $(\boldsymbol{X}^0, V^0, \kappa^0)$ to be prescribed. To choose a consistent initial approximation of the normal velocity $V^0$ and the mean curvature $\kappa^0$, we use a smooth curve for our initial curve $\Gamma_0$. So that its mean curvature $\kappa_0=-\boldsymbol{n}\cdot \partial_{ss}\boldsymbol{X}_0$ and velocity $V_0=\partial_{ss}\kappa_0+\frac{1}{2}\kappa_0^3$ can be given analytically. And the initial approximation $(\boldsymbol{X}^0, \kappa^0, V^0)$ in our iterative solver \eqref{eq: Newton} can be thus determined. Moreover, the iteration tolerance is chosen as $\text{tol} = 10^{-12}$.

First we aim to perform a convergence test for our energy-stable PFEM \eqref{thm: energy, full}. Here we choose the time step size $\tau$ and the mesh size $h$ as $\tau = \frac{h^2}{2}$ in all cases. This special choice is based on \eqref{thm: energy, full} being first-order in time and second-order in space, and can thus achieve both accuracy and computational efficiency.

We consider three distinct initial curves for our tests:
\begin{enumerate}
\item a unit circle: $\Gamma:=x^2+y^2=1$,
\item an ellipse: $\Gamma:=\frac{x^2}{2}+y^2=1$, and
\item a 3-fold curve: 
\begin{align*}
\left\{\begin{array}{l}x=(1+\frac{1}{15}\cos (3\theta))\cos\theta,\\[0.3em]
y = (1+\frac{1}{15}\cos(3\theta))\sin\theta,
\end{array}\right.\qquad\theta\in[0,~2\pi].
\end{align*}
\end{enumerate}

To test the convergence rate for $V$ and $\kappa$, we employ the first initial shape, the unit circle. The exact solution of the unit circle is known as follows \cite{Barrett08}:
\begin{equation}\label{eq: exact}
    \boldsymbol{X}(\rho, t)=R(t)\boldsymbol{X}_0(\rho),\quad  \kappa(\rho, t)=R(t)^{-1}, \quad V(\rho, t) = \frac{1}{2} R(t)^{-3},\quad \forall \rho \in \mathbb{T},\ t\geq 0,
\end{equation}
with $R(t)=(1+2t)^{1/4}$. We test our energy-stable PFEM \eqref{thm: energy, full} with a uniformly partitioned initial approximation $\boldsymbol{X}^0(\rho_j)=(\cos (2\pi \rho_j), \sin (2\pi\rho_j))^T, \forall 0\leq j\leq N$;
and a non-uniformly partitioned initial approximation: $\boldsymbol{X}^0(\rho_j)=(\cos(2\pi\rho_j+0.1\sin(2\pi\rho_j)), \sin(2\pi\rho_j+0.1\sin(2\pi\rho_j)))^T, \forall 0\leq j\leq N$. The latter one is the same as that given in \cite{Barrett08}. We apply the $L^{\infty}$-norm to assess the numerical error between the numerical solution $(\boldsymbol{X}^m, V^m, \kappa^m)$ and the exact solution $(\boldsymbol{X}_e(t_m), V_e(t_m), \kappa_e(t_m))$. In Table 5.1, we report on the numerical errors between the numerical and exact solutions for the Willmore flow of a circle at $t_m=1$. It can be observed that the order of convergence for $V$ and $\kappa$ is almost 2.
\begin{table}[htp!]
\label{tb: convergence rate 1}
\begin{center}
\begin{tabular}{|c|c|c|c|c|}
\hline
& \multicolumn{4}{c|}{The uniformly partitioned unit circle} \\
\hline
$h$ & $\|V^m-V_e(1)\|_{L^\infty}$ & order & $\|\kappa^m - \kappa_e(1)\|_{L^\infty}$ & order \\
\hline
$2^{-3}$ & $1.56\times 10^{-03}$ & - & $2.51\times 10^{-02}$ & $-$ \\
$2^{-4}$ & $4.97\times 10^{-04}$ & $1.65 $& $6.01\times 10^{-03}$ & $2.06$ \\
$2^{-5}$ & $1.30\times 10^{-04}$ & $1.93 $& $1.49\times 10^{-03}$ & $2.02$ \\
$2^{-6}$ & $3.29\times 10^{-05}$ & $1.98 $& $3.70\times 10^{-04}$ & $2.00$ \\
$2^{-7}$ & $8.26\times 10^{-06}$ & $2.00 $& $9.25\times 10^{-05}$ & $2.00$ \\
$2^{-8}$ & $2.07\times 10^{-06}$ & $2.00 $& $2.31\times 10^{-05}$ & $2.00$ \\
\hline
\end{tabular}

\vspace{2em}

\begin{tabular}{|c|c|c|c|c|}
\hline
& \multicolumn{4}{c|}{The non-uniformly partitioned unit circle} \\
\hline
$h$ & $\|V^m-V_e(1)\|_{L^\infty}$ & order & $\|\kappa^m - \kappa_e(1)\|_{L^\infty}$ & order \\
\hline
$2^{-3}$ & $1.79\times 10^{-03}$ & $-$ & $2.84\times 10^{-02}$ &  - \\
$2^{-4}$ & $5.52\times 10^{-04}$ & $1.70$ & $6.84\times 10^{-03}$ & $2.06 $\\
$2^{-5}$ & $1.44\times 10^{-04}$ & $1.94$ & $1.69\times 10^{-03}$ & $2.01 $\\
$2^{-6}$ & $3.64\times 10^{-05}$ & $1.99$ & $4.22\times 10^{-04}$ & $2.00 $\\
$2^{-7}$ & $9.11\times 10^{-06}$ & $2.00$ & $1.06\times 10^{-04}$ & $2.00 $\\
$2^{-8}$ & $2.28\times 10^{-06}$ & $2.00$ & $2.64\times 10^{-05}$ & $2.00 $\\
\hline
\end{tabular}
\end{center}
\caption{Numerical errors $\|V^m-V_e(1)\|_{L^\infty}$, $\|\kappa^m - \kappa_e(1)\|_{L^\infty}$ and their convergence order with the uniformly partitioned unit circle and the non-uniformly partitioned unit circle.}
 \end{table}

We adopt the manifold distance $M(\cdot, \cdot)$ given in \cite{zhao2021energy} to test the convergence order for curves. Let $\Omega_1$ and $\Omega_2$ be the inner enclosed regions by $\Gamma_1$ and $\Gamma_2$, respectively,
then the manifold distance $M(\Gamma_1, \Gamma_2)$ is given as
\begin{equation}
     M(\Gamma_1, \Gamma_2):=2|\Omega_1\cup \Omega_2|-|\Omega_1|-|\Omega_2|,
 \end{equation}
where $|\Omega|$ denotes the area of $\Omega$. The numerical error $e^h(t_m)$ is $M(\Gamma^m, \Gamma_e(t_m))$, where $\Gamma_e(t_m):=\boldsymbol{X}_e(\mathbb{T})$ denotes the exact curve. For the unit circle, the exact curve is given in \eqref{eq: exact}. While for the ellipse and the 3-fold curve, the exact curves are not analytically given, and we utilize a very fine mesh size $h_e = 2^{-9}$ and time step size $\tau_e = \frac{h_e^2}{2}$ to approximate the exact curve $\Gamma_e$. 

\begin{figure}[htp!]
\centering
\includegraphics[width=0.5\textwidth]{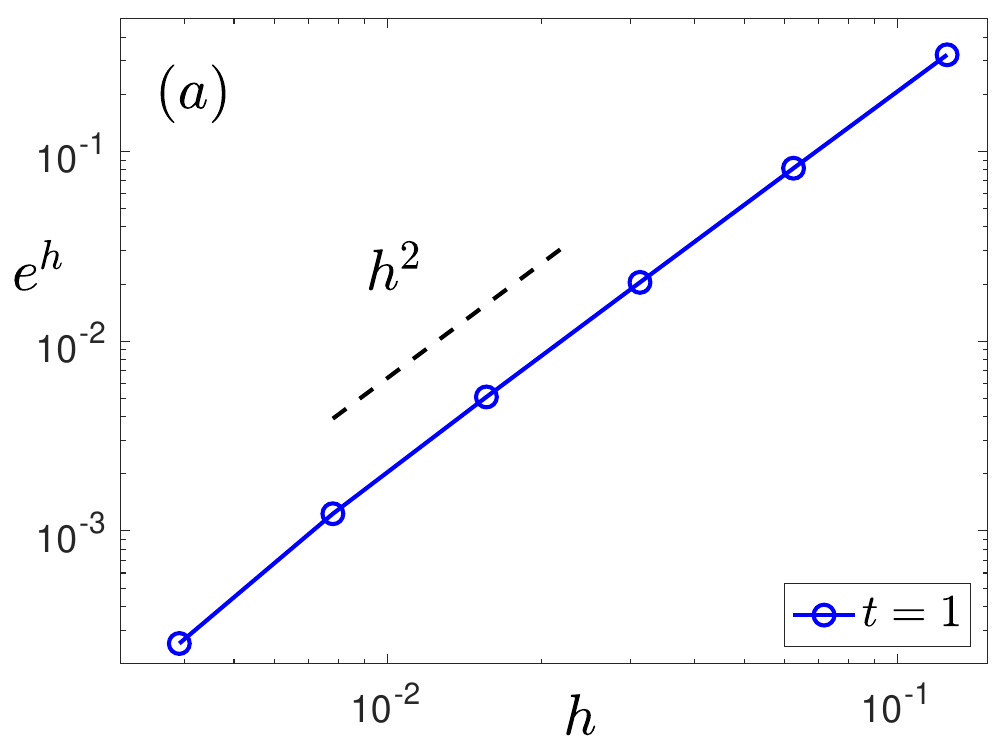}\includegraphics[width=0.5\textwidth]{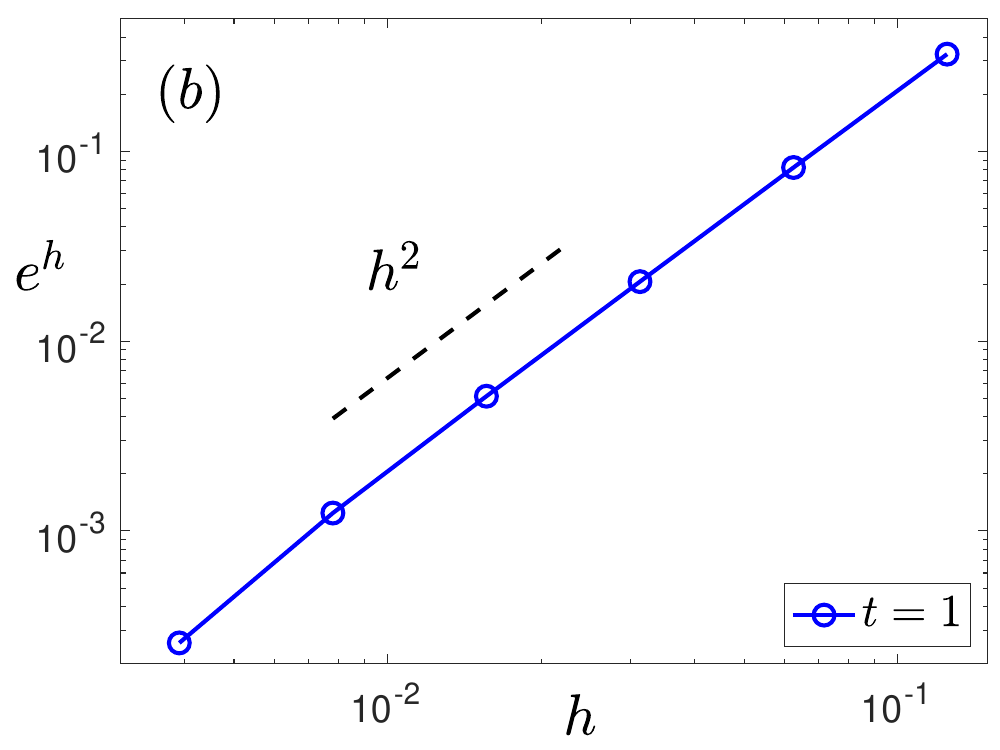}
\includegraphics[width=0.5\textwidth]{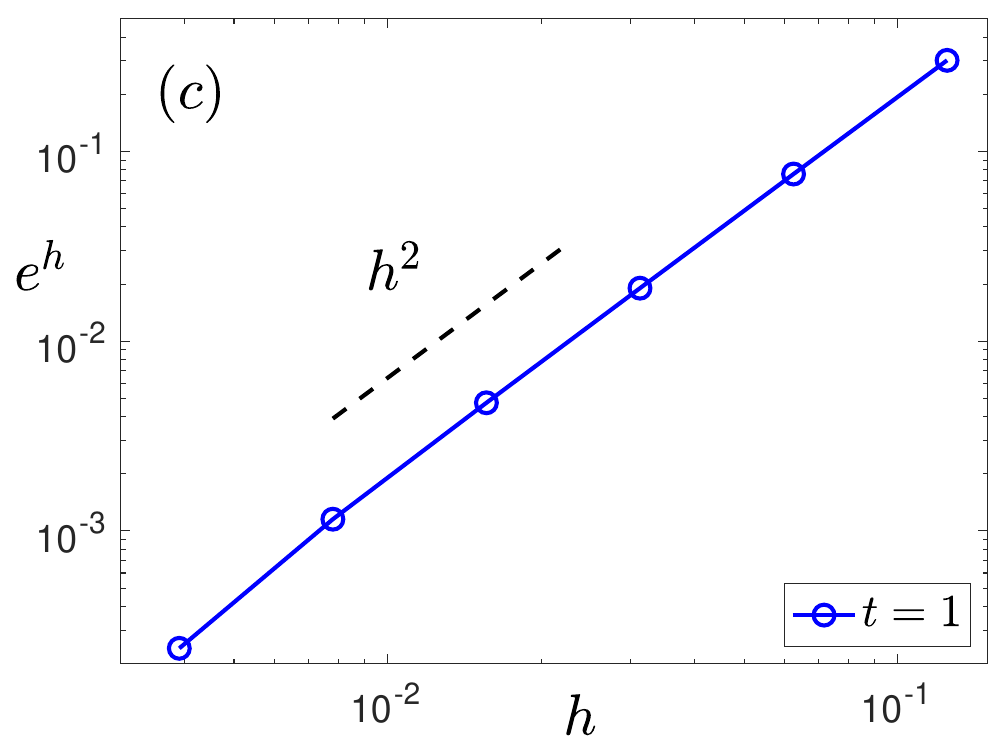}\includegraphics[width=0.5\textwidth]{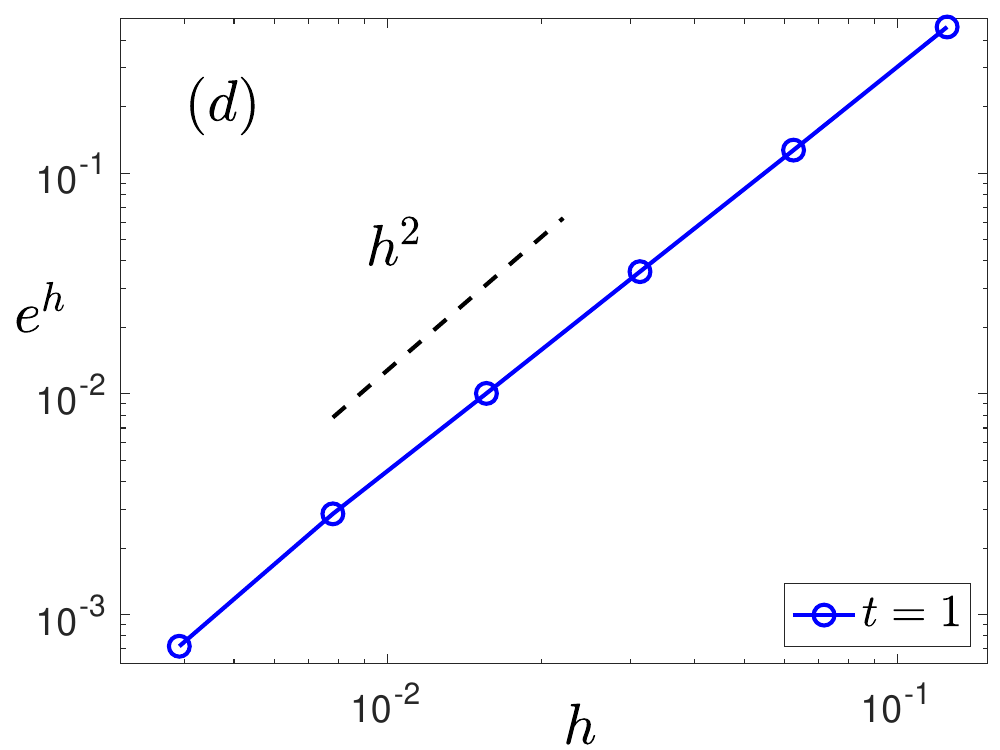}
\caption{Plot of numerical errors $e^h$ of different initial curves at $t_m=1$: (a) a uniformly partitioned unit circle, (b) a non-uniformly partitioned unit circle, (c): an ellipse, (d): a 3-fold curve}
\label{fig: convergence}
\end{figure}

Figure \ref{fig: convergence} demonstrates the errors $e^h$ between the numerical and exact curve for the Willmore flow of the ellipse and the 3-fold curve at $t_m=1$. The observed convergence order is almost 2 among different initial curves, indicating that the second-order convergence rate is robust. 

%To further investigate the performance of \eqref{thm: energy, full}, we focus on two indicators: the number of iterations required for each step, and the mesh ratio $\Psi(t)$. The mesh ratio is quantified as follows:
%\begin{equation}
%\Psi(t_m)=\Psi(\Gamma^m):=\frac{\max_{1\leq j\leq N} |\boldsymbol{h}_j^m|}{\min_{1\leq j\leq N}|\boldsymbol{h}_j^m|}.
%\end{equation}

To further investigate the performance of \eqref{thm: energy, full}, we test the number of iterations required for each step. Figure \ref{fig: cnt} presents the number of iterations with the previous four different setups, using $h = 2^{-8}$ and $\tau=\frac{h^2}{2}$. We observe that the number of iterations is mostly within two.  

\begin{figure}[htp!]
\centering
\includegraphics[width=0.49\textwidth]{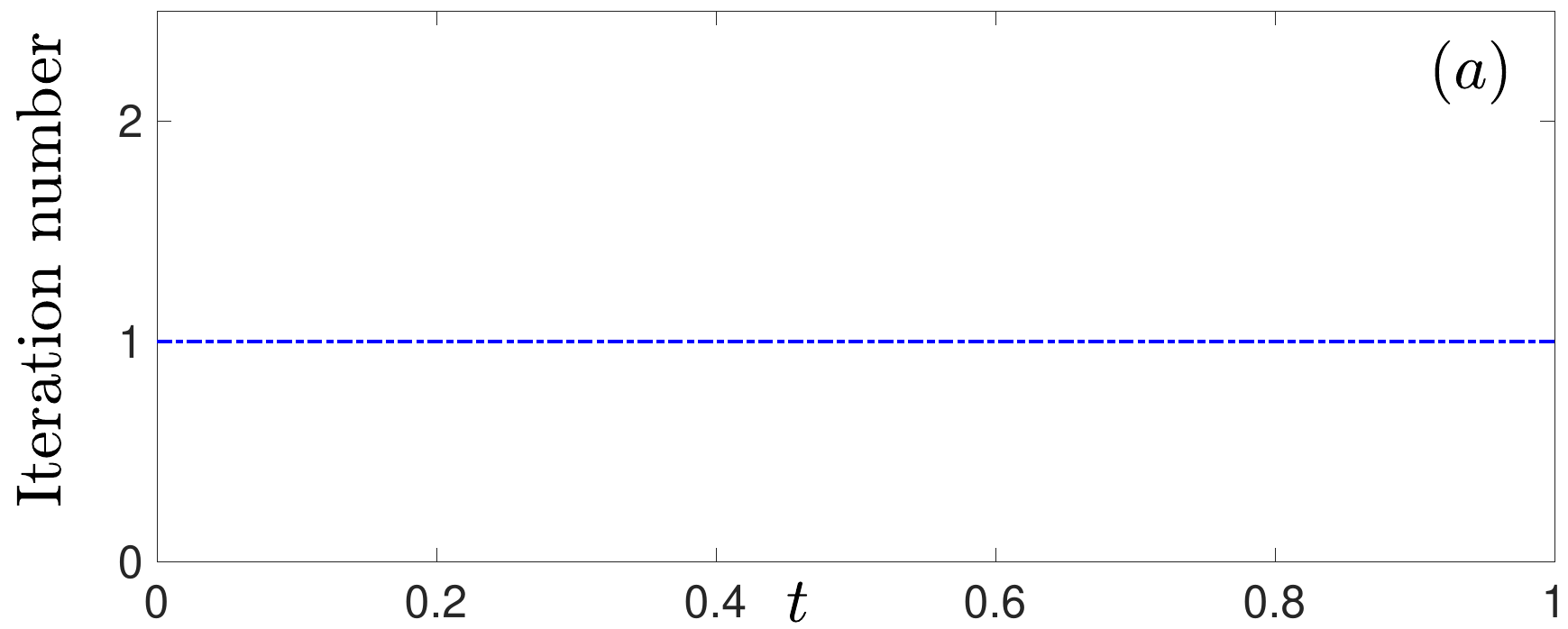}\;\includegraphics[width=0.49\textwidth]{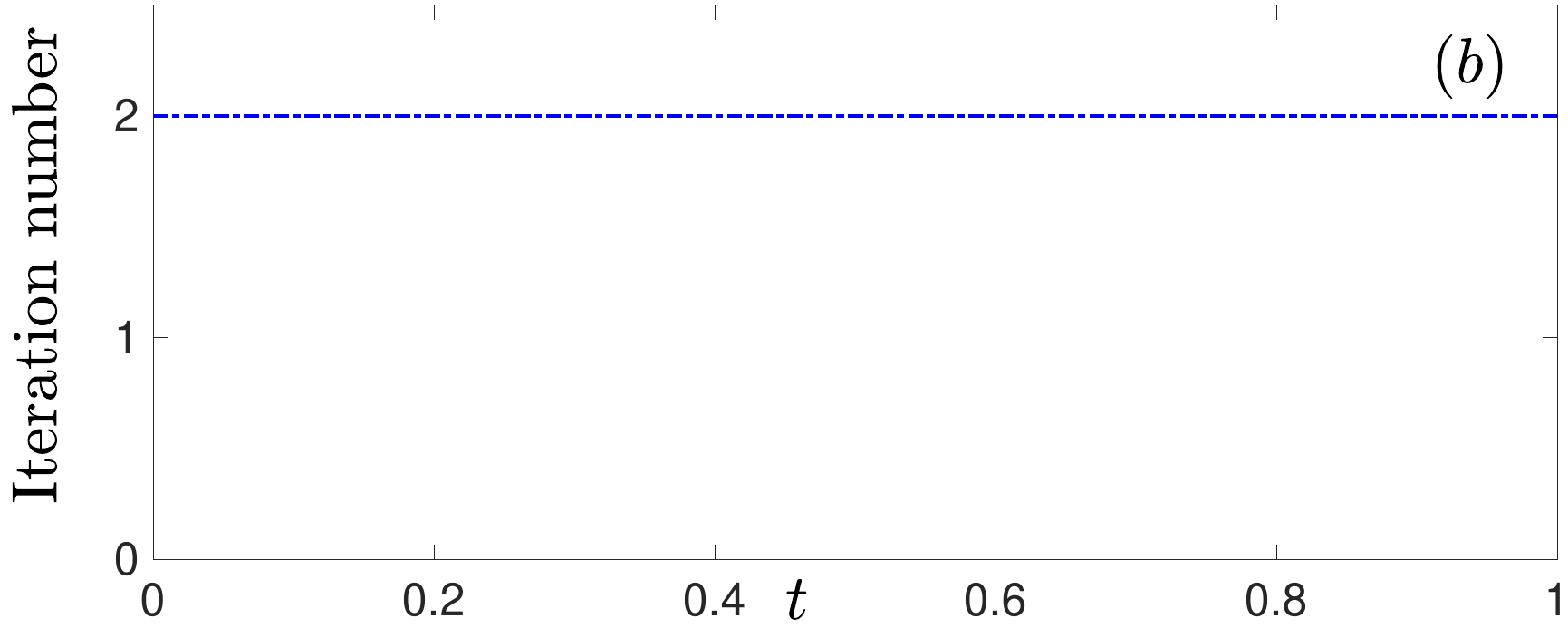}
\includegraphics[width=0.49\textwidth]{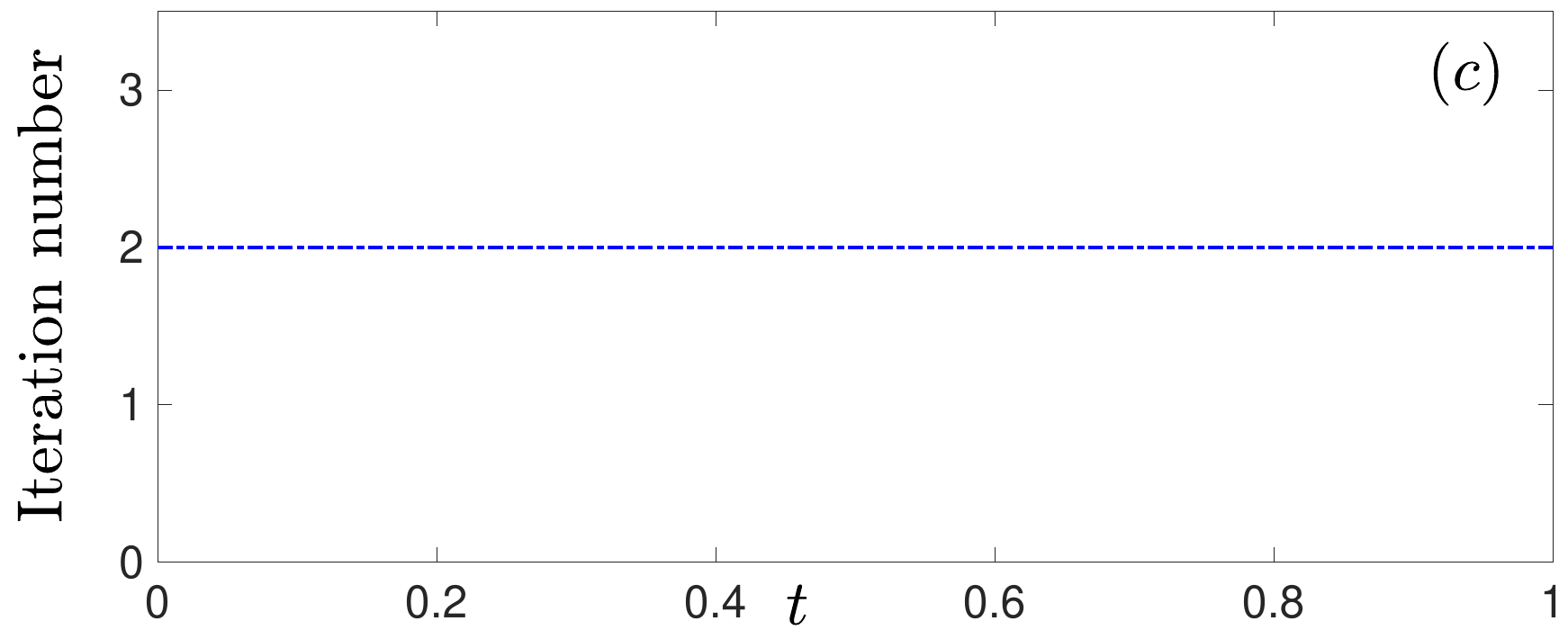}\;\includegraphics[width=0.49\textwidth]{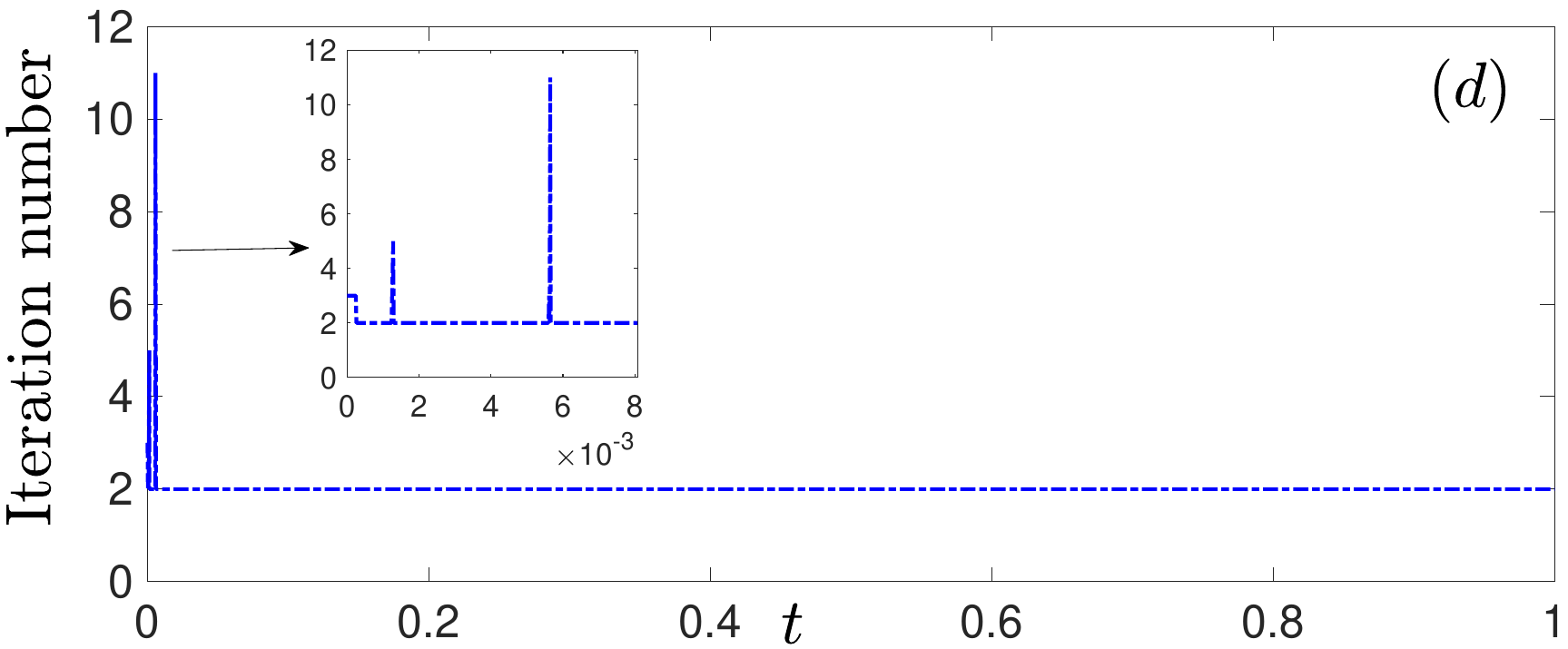}
\caption{Plot of the number of iterations of different initial curves with $h=2^{-8}$ and $\tau=\frac{h^2}{2}$ at $t_m=1$: (a) a uniformly partitioned unit circle, (b) a non-uniformly partitioned unit circle, (c): an ellipse, (d): a 3-fold curve.}
\label{fig: cnt}
\end{figure}

Next, we test the unconditional energy stability. In Figure \ref{fig: energy} (a), (b), we illustrate the normalized energy $W^m/W^0$ with fixed mesh size $h = 2^{-8}$ and varying time step size $\tau = 0.0001, 0.001, 0.005$ for the elliptical curves, and the 3-fold curve, respectively. We find that the energy is decreasing monotonically, and the unconditional energy stability holds for relatively large $\tau$. This result aligns well with our theoretical result of unconditional energy stability Theorem \ref{thm: energy, full}. 

\begin{figure}[htp!]
\centering
\includegraphics[width=0.5\textwidth]{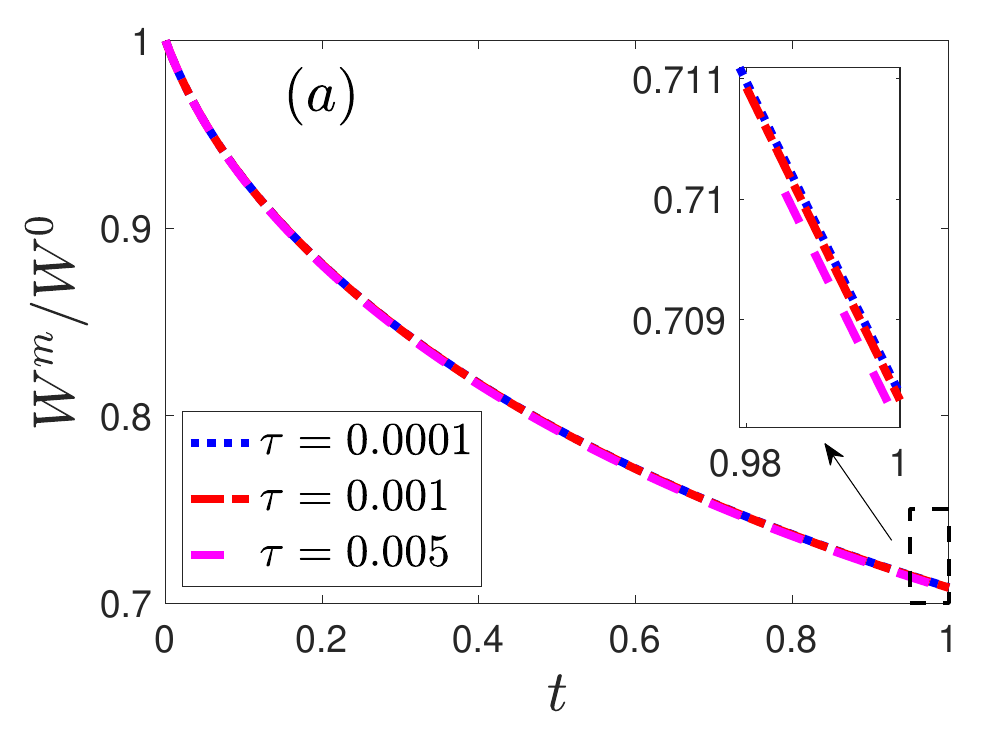}\includegraphics[width=0.5\textwidth]{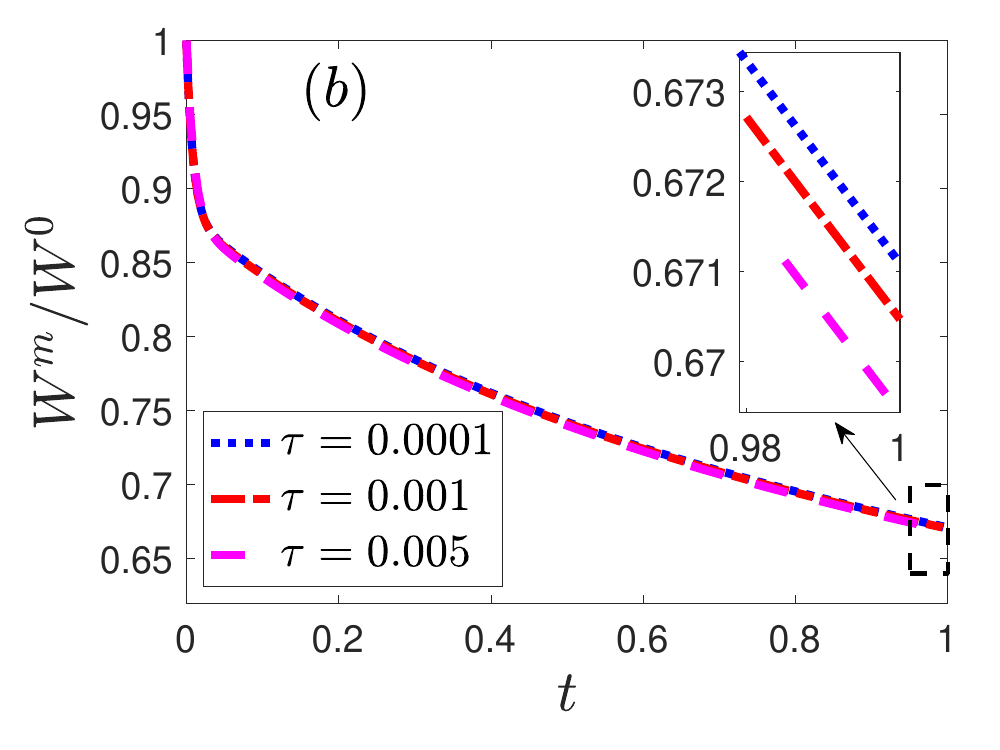}
\caption{Plot of the normalized energy $W^m/W^0$ with fixed mesh size $h = 2^{-8}$ and varying time step size $\tau = 0.0001, 0.001, 0.005$ for (a) the elliptical curves, and (b) the 3-fold curve.}
\label{fig: energy}
\end{figure}

%To further investigate the performance of the proposed energy-stable PFEM, we consider the number of iterations in the energy-stable PFEM and also the mesh ratio indicator given as
%\begin{equation}
 %   \Psi(\Gamma)=\frac{\max |\boldsymbol{h}_i|}{\min|\boldsymbol{h}_i|}.
%\end{equation}

%Figure 3, 4, shows the two indicators for the different four setups with $h = 0$ and $\tau = 0$. We can infer that, it only takes one or two iterations for each step, thus our numerical scheme is efficient. Also, we observe that for the uniformly partitioned case, the mesh ratio remains 1. Even for the non-uniformly partitioned cases, the mesh ratio still behaves in a disired way.

%Next, we validate the unconditionally energy stability. In figure 5(a), 5(b), we choose different $\tau$ and $h$ for the elliptical curves, and the 3-fold curve, respectively. We observe that the energy is decreasing monotonically as expected. Also, in figure 6(a), 6(b) we fix the mesh size $h$ and varies the time step size $\tau$. We find that the energy stability still holds for relatively large $\tau$. For a fourth-order equation, we can even choose $\tau = O(h)$, this is the advantage of our energy-stable numerical scheme.

We would also like to illustrate the evolution of planar curves under the Willmore flow. Figure \ref{fig: evolvee} (a), (b) depict the evolution for the initial elliptical curves and the 3-fold curve with $h=2^{-5}$ and $\tau = 0.001$, respectively. It can be observed that the initial shapes gradually transform into a circle and then expand uniformly. The rate of expansion decreases as the circle's radius increases. We have also noted that, the mesh points remain uniformly distributed along the curve, indicating the desired mesh quality of our energy-stable PFEM \eqref{eq: willmore full}.

\begin{figure}[htp!]
\centering
\includegraphics[width=0.6\textwidth]{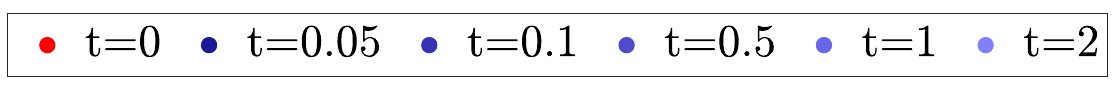}
\includegraphics[width=0.5\textwidth]{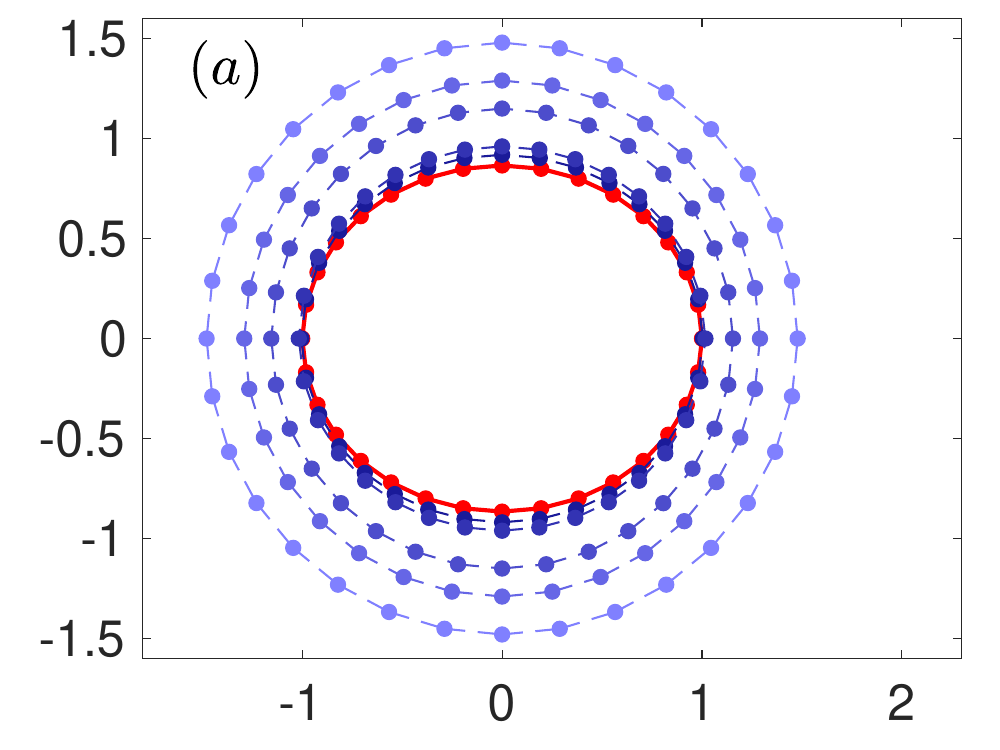}\includegraphics[width=0.5\textwidth]{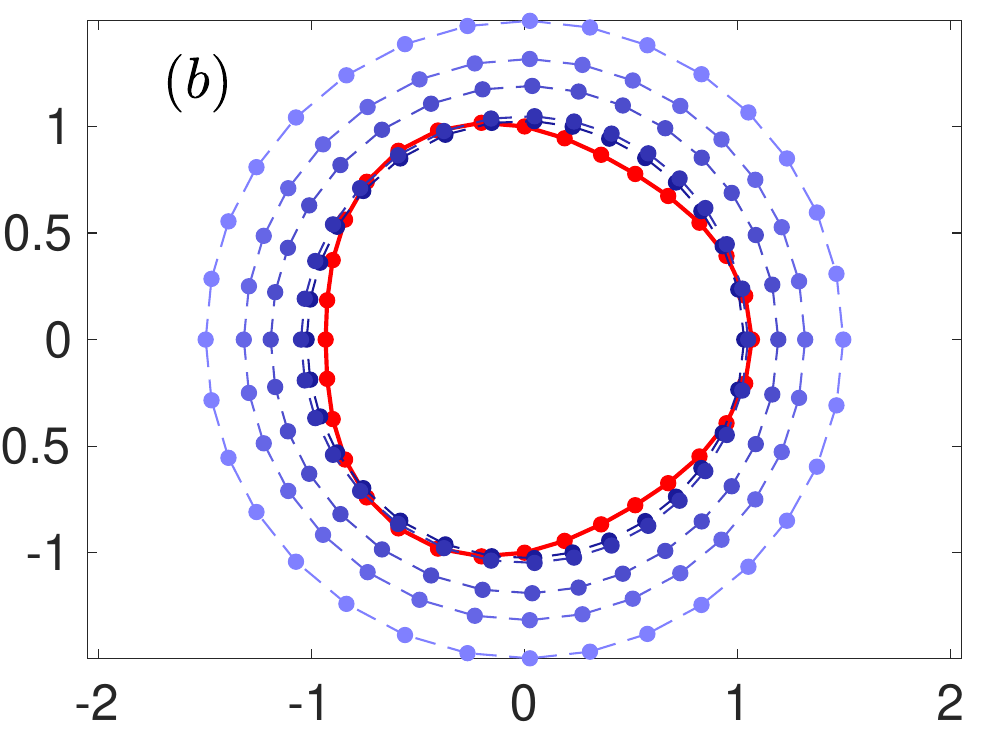}
\caption{The evolution of (a) an initial elliptical curve (b) an initial 3-fold curve under the Willmore flow, where the initial curve is colored in red, and the intermediate curves evolve over time from dark blue to light blue. The mesh size $h = 2^{-5}$ and the time step size $\tau = 0.001$.}
\label{fig: evolvee}
\end{figure}

Finally, we want to remark that, unlike the PFEM for surface diffusion in \cite{Barrett07}, the mesh points are not guaranteed to be asymptotically equally distributed in our energy-stable PFEM \eqref{eq: willmore full} for the Willmore flow. Therefore, for complex curves, our numerical scheme may require mesh redistribution. However, there is no mesh relocation methods that can preserve the energy dissipation. So it would be interesting to explore how to extend our numerical scheme to complex curves by improving the mesh quality as in \cite{barrett2017stable}. 

\section{Conclusions}

We have developed an energy-stable parametric finite element method (PFEM) for the Willmore flow of planar curves in $\mathbb{R}^2$. This advancement was achieved by introducing two novel geometric identities, the first one couples the outward unit normal vector $\boldsymbol{n}$ and the normal velocity $V$ together, and the second one considers the time derivative of $\kappa$. Based on these geometric identities, we proposed a new geometric PDE for the Willmore flow and derived its corresponding weak formulation. We employed piecewise linear elements for spatial discretization and the backward Euler method for time discretization, resulting in an implicit energy-stable PFEM full-discretization solved via the Newton-Raphson iteration. We further proved its unconditional energy stability at the full-discretized level. The proposed energy-stable PFEM demonstrated a robust second-order convergence rate in $L^\infty$-norm for $V$ and $\kappa$, and in manifold distance for the curve $\Gamma$. Furthermore, it exhibited good mesh quality, with most iterations converging within 2 iterations for solving the nonlinear system at each time level. These theoretical and numerical results indicate the efficiency and accuracy of our proposed energy-stable PFEM in simulating the Willmore flow of planar curves.

In the future, we will further extend the two novel geometric identities to other geometric flows related to the Willmore energy, especially the Helfrich flow. We would also like to explore the artificial tangential velocities to further improve the mesh quality.

%===========================================================Acknowledgement=========================================================
%===================================================================================================================================

\bigskip

%===========================================================Acknowledgement=========================================================
%===================================================================================================================================
\section*{Acknowledgement}
This work was partially supported by the Ministry of Education of Singapore under its AcRF Tier 2 funding MOE-T2EP20122-0002 (A-8000962-00-00). Part of the work was done when the authors were visiting the Institute of Mathematical Science at the National University of Singapore in 2023.

%\begin{acknowledgements}
%If you'd like to thank anyone, place your comments here
%and remove the percent signs.
%\end{acknowledgements}

% Authors must disclose all relationships or interests that 
% could have direct or potential influence or impart bias on 
% the work: 
%
% \section*{Conflict of interest}
%
% The authors declare that they have no conflict of interest.

% BibTeX users please use one of
\bibliographystyle{siamplain}
\bibliography{thebib}     % name your BibTeX data base

\begin{comment}

\end{comment}

\end{document}